\documentclass[a4paper,12pt]{article}
\usepackage[top=2cm, bottom=2cm, outer=2cm, inner=2cm, headsep=14pt]{geometry}
\usepackage{amsmath,amsfonts}
\usepackage{enumitem}
\usepackage{adjustbox}
\usepackage{blkarray,empheq}

\usepackage{multicol,color} 

\usepackage{tikz} 
\usetikzlibrary{decorations.pathreplacing,decorations.markings} 
\usepackage{url} 

\usepackage{soul,xcolor} 

\tikzset{
  on each segment/.style={
    decorate,
    decoration={
      show path construction,
      moveto code={},
      lineto code={
        \path [#1]
        (\tikzinputsegmentfirst) -- (\tikzinputsegmentlast);
      },
      curveto code={
        \path [#1] (\tikzinputsegmentfirst)
        .. controls
        (\tikzinputsegmentsupporta) and (\tikzinputsegmentsupportb)
        ..
        (\tikzinputsegmentlast);
      },
      closepath code={
        \path [#1]
        (\tikzinputsegmentfirst) -- (\tikzinputsegmentlast);
      },
    },
  },
  mid arrow/.style={postaction={decorate,decoration={
        markings,
        mark=at position .5 with {\arrow[#1]{stealth}}
      }}},
}

\usepackage{avant}

\allowdisplaybreaks

\newcommand{\G}{\Gamma}
\def\nnu{{\boldsymbol\nu}}
\newcommand{\Mat}{\mbox{\rm Mat}}

\def\jj{{\boldsymbol j}}

\def\ww{{\boldsymbol w}}

\def\uu{{\boldsymbol u}}
\def\O{{\boldsymbol O}}

\def\0{{\boldsymbol 0}}

\def\aalpha{{\boldsymbol\alpha}}
\def\CC{{\mathbb C}}
\def\QQ{{\mathbb Q}}

\def\M{{\cal M}}

\def\E{{\cal E}}
\def\F{{\cal F}}

\def\RR{{\mathbb R}}
\def\FF{{\mathbb F}}
\def\NN{{\mathbb N}}
\def\ZZ{{\mathbb Z}}
\def\A{{\mathcal A}}

\def\S{{\mathcal S}}
\def\R{{\mathcal R}}
\def\P{{\mathcal P}}
\def\B{{\mathcal B}}
\def\M{{\mathcal M}}

\def\ol{\overline}
\def\ds{\displaystyle}

\def\Lra{\Leftrightarrow}
\def\Ra{\Rightarrow}
\def\la{\leftarrow}
\def\La{\Leftarrow}
\def\ra{\rightarrow}
\def\XXi{{\mathfrak X}}

\DeclareMathOperator{\Span}{span}

\DeclareMathOperator{\ii}{{\boldsymbol{i}}}
\DeclareMathOperator{\hh}{{\boldsymbol{h}}}
\DeclareMathOperator{\oo}{{\boldsymbol{0}}}

\DeclareMathOperator{\im}{im}

\DeclareMathOperator{\spec}{spec}

\newtheorem{theorem}{Theorem}[section]
\newtheorem{lemma}[theorem]{Lemma}
\newtheorem{corollary}[theorem]{Corollary}
\newtheorem{proposition}[theorem]{Proposition}

\newtheorem{remark}[theorem]{Remark}

\newtheorem{researchProblem}[theorem]{Research problem}

\newtheorem{problem}{Problem}[section]

\definecolor{ForestGreen}{RGB}{12, 110, 46}
\definecolor{ForestGreenTwo}{RGB}{120, 110, 86}

\newenvironment{proof}{{\noindent\it Proof. }}{\nopagebreak\hspace*{0.5cm}\hfill$\hbox{\rule{3pt}{6pt}}$\smallskip}

\definecolor{ForestGreen}{RGB}{12, 110, 46}
\definecolor{ForestGreenTwo}{RGB}{120, 110, 86}

\newfont\fiverm{cmr5}
\def\eeq{\end{equation}} 
\def\lbeq#1{\begin{equation} \label{#1}} 
 
\title{
On commutative association schemes and associated (directed) graphs
}

\author{
{Giusy Monzillo}\\
{\small Faculty of Mathematics, Natural Sciences}\\
{\small and Information Technologies}\\
{\small University of Primorska}\\
{\small Muzejski trg 2, 6000 Koper, Slovenia }\\
{\small Giusy.Monzillo@famnit.upr.si} \and
{Safet Penji\'c}\\
{\small Faculty of Mathematics, Natural Sciences}\\
{\small and Information Technologies, and}\\
{\small Andrej Maru\v{s}i\v{c} Institute}\\
{\small University of Primorska}\\
{\small Muzejski trg 2, 6000 Koper, Slovenia }\\
{\small Safet.Penjic@iam.upr.si}
}

\begin{document}
\setstcolor{red}
\maketitle

\begin{abstract}
Let $\M$ denote the Bose--Mesner algebra of a commutative $d$-class association scheme $\XXi$ (not necessarily symmetric), and $\G$ denote a (strongly) connected (directed) graph with adjacency matrix $A$. Under the assumption that $A$ belongs to $\M$, we describe the combinatorial structure of $\G$. Moreover, we provide an algebraic-combinatorial characterization of $\G$ when $A$ generates $\M$.

Among else, we show that, if $\XXi$ is a commutative $3$-class association scheme that is not an amorphic symmetric scheme, then we can always find a (directed) graph $\G$ such that the adjacency matrix $A$ of $\G$ generates the Bose--Mesner algebra $\M$ of $\XXi$.
\end{abstract}


\smallskip
{\small

}


\section{Introduction}
\label{1a}

In this paper, we study connections between commutative association schemes and (directed) graphs, by considering the following question: when can a commutative association scheme be generated by a (directed) graph? Formal definitions and an introduction to the theory of commutative association schemes are given in Section~\ref{ca}.

Let $\M$ denote the Bose--Mesner algebra of a commutative $d$-class association scheme $\XXi=(X,\R)$ (note that $\M$ does not need to be symmetric). To give a motivation and an introduction to our problem, in the next few lines, we first show that $\M$ is a monogenic algebra, that is, we show that there always exists a matrix $A\in\Mat_X(\CC)$ which generates $\M$, i.e., $\M=(\langle A\rangle,+,\cdot)$ (we say that a matrix $A$ {\em generates} $\M$ if every element in $\M$ can be written as a polynomial in $A$). Since $\M$ is a space of commutative normal matrices, from a well-known result on commutative sets of normal matrices, there exists a unitary matrix $U\in\Mat_X(\CC)$ which diagonalizes $\M$: to each $B\in\M$ 
there corresponds a diagonal matrix $\Lambda\in\Mat_X(\CC)$ such that $B=U \Lambda \ol{U}^\top$, and the diagonal entries of $\Lambda$ are the eigenvalues of $B$. When $B$ runs through $\M$, the matrix $\Lambda$ runs through a subalgebra $\F$ of $\Mat_X(\CC)$ which is isomorphic to $\M$. An explicit isomorphism $\psi:\M\ra\F$ is given by $\psi(B)=\ol{U}^\top B U$. Since $\M$ is an algebra of dimension $d+1$, the algebra $\F$ is also of dimension $d+1$. Moreover, there exists a set of diagonal $01$-matrices $\{F_i\}_{0\le i\le d}$ which is a basis of $\F$. Then, for arbitrary non-zero pairwise distinct complex scalars $\alpha_i$ $(0\le i \le d)$, the matrix $F=\sum_{i=0}^d\alpha_i F_i$ generates $\F$, i.e., $\F=(\langle F\rangle, +,\cdot)$. Thus, the matrix $A=\psi^{-1}(F)$ generates $\M$ and has $d+1$ distinct eigenvalues (for details we refer to Sections~\ref{ca} and \ref{da}). A reader more familiar with the field of association schemes will notice that our claim on the existence of a generator $A$ of $\M$ also follows from the proof of \cite[Theorem~2.2]{DP}. For a different approach in proving that the Bose--Mesner algebra $\M$ of an arbitrary commutative $d$-class association scheme $\XXi$ (which is not necessarily symmetric) can be generated by $A$, see Lemma~\ref{dk} in Section~\ref{da}. In this paper, we are interested in the following problem.

\begin{problem}
\label{aa}{\rm
When can the Bose--Mesner algebra $\M$ of commutative $d$-class association scheme $\XXi$ (which is not necessarily symmetric) be generated by a $01$-matrix $A$? In other words, for a given $\XXi$, under which combinatorial and algebraic restrictions can we find a $01$-matrix $A$ such that $\M=(\langle A\rangle, +, \cdot)$? Moreover, since such a matrix $A$ is the adjacency matrix of some (directed) graph $\G$, can we describe the combinatorial structure of $\G$? The vice-versa question is also of importance, i.e., what combinatorial structure does a (directed) graph need to have so that its adjacency matrix will generate the Bose--Mesner algebra of a commutative $d$-class association scheme $\XXi$?
}\end{problem}

In the case when $\XXi$ is a symmetric association scheme, our problem is connected with quotient-polynomial graphs (undirected graphs which generate symmetric association schemes, see \cite{FQpG,FMPS}). Recall that in \cite{FQpG} {\sc Fiol} defined a quotient-polynomial graph a little bit differently, that is, the author defined the quotient-polynomial graph as a graph $\G$ with vertex set $X$ for which the adjacency matrices of a walk-regular partition of $X\times X$ belong to the adjacency algebra of $\G$. Then, the author described some algebraic properties of such graphs and proved that $\G$ is the connected generating graph of an association scheme $\XXi$ if and only if $\G$ is a quotient-polynomial graph. Following this paper, in \cite{FMPS} {\sc Fiol} and {\sc Penjić} revisited this topic from another point of view, finding some additional algebraic properties as well as describing the combinatorial structure of quotient-polynomial graphs. In both cited papers \cite{FQpG,FMPS}, the authors studied the case of undirected graphs and with it the case of a symmetric (adjacency) algebra. In this paper, we study commutative association schemes (not necessarily symmetric) and, as a by-product, we also get some interesting results for symmetric association schemes. More precisely we answer the following question: {\em Is it possible that every symmetric association scheme is generated by some (quotient-polynomial) graph?} (The answer for a $3$-class association scheme is given in Theorem~\ref{ph}.) 

In the case when $\XXi$ is a \underline{symmetric} $3$-class association scheme, by the result of {\sc Van Dam} in \cite[Theorem~5.1]{vDe3} together with our Lemma~\ref{dk}, 
we get partial answers to questions posted in Problem~\ref{aa}. 
For the moment, let $\G$ denote a connected regular graph with 4 distinct eigenvalues and adjacency matrix $A$. In \cite[Theorem~5.1]{vDe3} {\sc Van Dam} proved that $A$ is one of the adjacency matrices of a $3$-class association scheme if and only if two adjacent vertices have a constant number of common neighbors, and the number of common neighbors of any two nonadjacent vertices takes precisely two values. In the same paper \cite{vDe3}, the author gave some answers about when and how to use the combinatorial structure of strongly-regular graphs to obtain a $3$-class association scheme (see, for example, \cite[Proposition~5.2]{vDe3}). In this paper, we fully describe when a $3$-class association scheme (not necessarily symmetric) can be generated by a graph.

We are interested in the natural problem of describing the combinatorial structure and algebraic properties of (directed) graphs which will generate commutative association schemes. This provides us with a different approach in finding new association schemes, using the structure of (directed) graphs. For example, if we pick some well-known family of undirected graphs, and give them an orientation on the edges that satisfy some of (if not all) properties described in this paper, will we get a candidate which generates a commutative association scheme? This paper gives some answers to this question too. 

Association schemes arise in group theory (see, for example, \cite{EsPi,HA,YM}), design theory (see, for example, \cite{BBBh,GaSsV,QjLj}), graph theory (see, for example, \cite{HKK,KM,TLy}), coding theory (see, for example, \cite{CP,DL,HFG}) and more (see, for example, \cite{CD,HAt,ItMa}). Some of the most well-studied association schemes are distance-regular graphs (see, for example, \cite{CGGV,DG,VJ}), including Moore graphs (see, for example, \cite{BH,NH,ZP}), distance-transitive graphs (see, for example, \cite{MT,RrJ,EY}), strongly-regular graphs (see, for example, \cite{SS,KPS,KSS}), generalized polygons \cite{ZP,BH}, etc. The problem of the construction association scheme is not a new one, and our paper is also going in this direction. Association schemes can be constructed from various kinds of objects, and we count just a few of them: from known association schemes (see, for example, \cite{HY}), from linear codes (see, for example, \cite{HCXL, Slc}), from Boolean bent functions (see, for example,  \cite{PTFK}), from cyclotomy over products of finite fields (see, for example, \cite{FGRM, MAG}), from $q$-polynomials (see, for example, \cite{GXc}), and so on. Moreover, association schemes are an important tool in the study of linear codes (\cite{WT, CGt}), difference sets (\cite{KMO}), objects from finite geometry (\cite{JGA,GA}), combinatorial designs (\cite{xB2}) and etc. For examples on how to construct self-orthogonal codes from association schemes, see, for example, \cite{DCSR,SRSA}.

We say that a (directed) graph $\G$ \emph{generates} a commutative association scheme $\XXi$ if and only if the adjacency matrix $A$ of $\G$ generates the Bose--Mesner algebra $\M$ of $\XXi$, and in symbols we write $\M=(\langle A\rangle,+,\cdot)$. Our main results are Theorems~\ref{ph}, \ref{pi} and \ref{PI}.

In Theorem~\ref{ph}, we characterize $3$-class amorphic symmetric schemes as the only commutative $3$-class association schemes which fails to satisfy the ``single-01-matrix generator'' property of Problem~\ref{aa}. In other words, except for amorphic symmetric association schemes, every $3$-class association scheme can be generated by a $01$-matrix $A$. This matrix is the adjacency matrix of a (strongly) connected (directed) graph $\G=\G(A)$, and has $4$ distinct eigenvalues.

\begin{theorem}
\label{ph}
Let $\XXi$ denote a $3$-class commutative association scheme. If $\XXi$ is not an amorphic symmetric scheme, then there exists a (strongly) connected (directed) graph $\G=\G(A)$ such that the following hold.
\begin{enumerate}[label=\rm(\roman*)]
\item The adjacency matrix $A$ of $\G$ has exactly $4$ distinct eigenvalues.
\item The adjacency matrix $A$ generates the Bose--Mesner algebra $\M$ of $\XXi$.
\end{enumerate} 
Moreover, the scheme $\XXi$ is generated by a graph if and only if it is not an amorphic symmetric scheme.
\end{theorem}



In Theorem~\ref{pi}, we describe the combinatorial structure of a graph which `lives' in a commutative association scheme. We can say that Corollary~\ref{we} of Theorem~\ref{pi} is in some sense a more general version of the result of {\sc Van Dam} given in \cite[Theorem~5.1]{vDe3} as it includes also \underline{non-symmetric} commutative $3$-class association schemes.

\begin{theorem}
\label{pi}
Let $\M$ denote the Bose--Mesner algebra of a commutative $d$-class association scheme $\XXi=(X,\R)$, and $A\in\M$ denote a $01$-matrix. Assume that $\G=\G(A)$ denotes a (strongly) connected (directed) graph. Then the following hold.
\begin{enumerate}[label=\rm(\roman*)]
\item For every vertex $x\in X$, there exists an $x$-distance-faithful intersection diagram (of an equitable partition $\Pi_x$) with $d+1$ cells.
\item The structure of the $x$-distance-faithful intersection diagram (of the equitable partition $\Pi_x$) from {\rm (i)} does not depend on $x$.
\end{enumerate}
\end{theorem}

Theorem~\ref{pi} is in some sense an extension of the result of {\sc Fiol} and {\sc Penjić} given in \cite[Theorem~4.1]{FMPS} as it extends to the family of directed graphs. With reference to Theorem~\ref{pi}, in Corollary~\ref{wg}, we consider the case when $\G$ generates $\XXi$.

Note that the result of Theorem~\ref{pi} is very general result that at first glance it seems like it is already known from literature. We did not manage to find something similar explicitly (or implicitly) written in literature. In Theorem~\ref{PI}, we give one of its applications by characterizing algebraic-combinatorial properties of $\G$ when $\G$ generates a commutative association scheme.

\begin{theorem}
\label{PI}
Let $\G=\G(A)$ denote a (strongly) connected (directed) graph with vertex set $X$, adjacency matrix $A$, $d+1$ distinct eigenvalues, and adjacency algebra $\A=\A(\G)$. Let 
$$
\Delta=\{(i,j) \mid i=\partial(x,y),\, j=\partial(y,x),\, x,y\in X \},
$$ 
and for any $\ii\in\Delta$ define $R_{\ii}=\{(x,y)\in X\times X\mid (\partial(x,y),\partial(y,x))=\ii \}$. Then, the following are equivalent.
\begin{enumerate}[label=\rm(\roman*)]
\item $\A$ is the Bose--Mesner algebra of a commutative $d$-class association scheme. 
\item $(X,\{R_{\ii}\}_{\ii\in\Delta})$ is a $d$-class association scheme.
\item $A$ is a normal matrix, $|\Delta|=d+1$, and the number of walks from $x$ to $y$ of every given length $\ell\ge 0$ only depends on the distances $\partial(x,y)$ and $\partial(y,x)$ (and do not depend on choice of the pair $(x,y)$).
\end{enumerate}
\end{theorem}

In Theorem~\ref{PI}, we give connections among weakly distance-regular graphs in sense of {\sc Comellas et al.} \cite{CFGM}, weakly distance-regular graphs in sense of {\sc Wang} and {\sc Suzuki} \cite{WS}, and commutative association schemes generated by a $01$-matrix $A$. We explain these two new notations (and their importance) in the next two paragraphs.

A directed graph $\G$ (of diameter $D$) is {\em weakly distance-regular in sense of} {\sc Comellas et al.} \cite{CFGM} if the number of walks of length $\ell$ $(0\le\ell\le D)$ in $\G$ between two vertices $x,y\in X$ only depends on $h=\partial(x,y)$ (on the distance from $x$ to $y$). In \cite[Theorem~2.2]{CFGM}, {\sc Comellas et al.} provided an algebraic-combinatorial characterization of such graphs. Moreover, in the same paper, the authors proved an equivalence between
(i) $A$ is a normal matrix and the set of distance-$i$ matrices $\{A_0,A_1,\ldots,A_D\}$ is a basis of the adjacency algebra $\A$ of $\G$;
and
(ii) there exist numbers $b_{ij}$ $(0\le i,j\le D)$ such that $|\G^\ra_1(y)\cap\G_j^\ra(x)|=b_{ij}$, for all $x\in X$, $y\in\G^\ra_i(x)$ $(0\le i,j\le D)$ (see \cite[Proposition~2.6]{CFGM}). Note the difference between above Theorem~\ref{PI}(iii) and the notion weakly distance-regular in sense of {\sc Comellas et al.} \cite{CFGM}. Among else, {\sc Comellas et al.} \cite{CFGM} studied the spectra of a  weakly distance-regular digraph and constructed several examples of such a graph. Our property (iii) of Theorem~\ref{PI} is restricted property of an open (up to our knowledge) research problem from \cite[Subsection 4.3]{CFGM}.
Some papers that are related with weakly distance-regular digraphs (in terms of number of walks of certain type) are \cite{CfFgM,GzJdZ,LL,DO}. We recommend papers \cite{FaMm,Ogr} for the study of spectrum of a (weakly distance-regular) directed graph.

For the moment, let $\G$ denote a directed graph with vertex set $X$, and consider the set $\Delta$ and relations $R_{\ii}$ $(\ii\in\Delta)$ from Theorem~\ref{PI}. A directed graph $\G$ is said to be {\em weakly distance-regular in sense of {\sc Wang} and {\sc Suzuki}} \cite{WS} if and only if $(X,\{R_{\ii}\}_{\ii\in\Delta})$ is a $|\Delta|$-class association scheme (for further insights regarding this definition, see paper of {\sc Suzuki} \cite{ShT}). In such a case, $\XXi(\G)$ is called the {\em attached scheme} of $\G$. In \cite{LGG,ShT,Wk,Yi,Yii}, some special families of weakly distance-regular digraphs in sense of {\sc Wang} and {\sc Suzuki} of small valency have been classified. Algebraic restrictions on weakly distance-regular digraphs called {\em thin}, {\em quasi-thin}, and {\em thick} were studied in \cite{ShT,Yqt,YWt}. Other papers that are directly or indirectly involved in the study of weakly distance-regular digraphs in sense of {\sc Wang} and {\sc Suzuki} are, for example, \cite{FW,FZL,GK,LS,Wki,WF,ZYW}. 

\bigskip
Our paper is organized as follows. In Section~\ref{ca}, we recall basic concepts from algebraic graph theory, including commutative association schemes and strongly-regular graphs (experts from the field can skip this section as well as Section~\ref{da}). Section~\ref{da} is a survey of all well-known properties that we use later in the paper: in this section we explain when a $01$-matrix generates the Bose--Mesner algebra of a scheme, and we explicitly (re)prove some results about the adjacency algebra of a (directed) graph, hidden in literature. Our paper then starts from Section~\ref{ra}. In Section~\ref{ra}, we prove Theorem~\ref{ph}. In Section~\ref{pa}, we prove Theorem~\ref{pi}, and we include several interesting corollaries of the claim.
In Section~\ref{6G}, we prove Theorem~\ref{PI}, among else by using the combinatorial structure of a (directed) graph, obtained in Theorem~\ref{pi}. We finish the paper with Section~\ref{ta}, where we describe possible further directions. {\color{blue} This manuscript has 39 pages, but actually the whole paper has 14 pages (see Section~\ref{ra}, \ref{pa} and \ref{6G}).}



\section{Preliminaries}
\label{ca}


{\color{blue}
{\bf Comment that we will delete from the final version of the paper:} 
Preliminary section (Section~\ref{ca}) has $10$ pages, which is usually too much for the paper. We expect that those readers who are familiar with the field will probably skip this section and start to read the paper from Section~\ref{da}. With a shorter version of this section, we found that the paper is not readable as we want it to be. 
}

A {\em directed graph} with {\em vertex set} $X$ and {\em arc set} $\E$ is a pair $\G=(X, \E)$ which consists of a finite set $X=X(\G)$ of {\em vertices} and a set $\E = \E(\G)$ of {\em arcs} ({\em directed edges}) between vertices of $\G$. As the initial and final vertices of an arc are not necessarily different, the directed graphs may have loops (arcs from a vertex to itself), and multiple arcs, that is, there can be more than one arc from each vertex to any other. If $e = (x,y)\in\E$ is an arc from $x$ to $y$, then the vertex $x$ (and the arc $e$) is {\em adjacent to} the vertex $y$, and the vertex $y$ (and the arc $e$) is {\em adjacent from} $x$. The {\em converse directed graph} $\ol{\G}$ is obtained from $\G$ by reversing the direction of each arc. 
For a vertex $x$, let $\G_1^{\la}(x)$ (and $\G_1^{\ra}(x)$) denote the set of vertices adjacent to (and from) the vertex $x$, respectively. In another words
$$
\G_1^{\ra}(x) = \{z\mid (x,z)\in \E(\G)\}
\qquad\mbox{ and }\qquad
\G_1^{\la}(x) = \{z\mid (z,x)\in \E(\G)\}.
$$
Two small comment about notation: (i) drawing directed edge from $x$ to $z$, we have $x\ra z$, which yields idea beyond using notation $\G_1^{\ra}(x)$; (ii) drawing directed edge from $z$ to $x$, we have $x\la z$ (or $z\ra x$), which yields idea beyond using notation $\G_1^{\la}(x)$. We abbreviate $\G_1(x)=\G_1^{\ra}(x)$. Also, instead of a set of vertices, we can consider a set of arcs: for a vertex $y$, let $D_1^{\la}(y)$ (and $D_1^{\ra}(y)$) denote the set of arcs adjacent to (and from) the vertex $y$, respectively. The number $|D_1^{\ra}(y)|$ we call the {\em out-degree of $y$} and is equal to the number of edges leaving $y$. The number $|D_1^{\la}(y)|$ we call the {\em in-degree of $y$} and is equal to the number of edges going to $y$. A directed graph $\G$ is {\em $k$-regular} if $|D^{\ra}_1(y)| = |D_1^{\la}(y)| = k$ for all $y\in X$.

\begin{figure}[t!]
\begin{center}
\begin{tikzpicture}[scale=.43]
\draw [line width=.8pt,draw=black,postaction={on each segment={mid arrow=black}}] (-9.72,2.6) -- (-3.14,7);
\draw [line width=.8pt,draw=black,postaction={on each segment={mid arrow=black}}] (-9.72,2.6) -- (0.42,-2.04);
\draw [line width=.8pt,draw=black,postaction={on each segment={mid arrow=black}}] (0.38,3.96) -- (-9.72,2.6);
\draw [line width=.8pt,draw=black,postaction={on each segment={mid arrow=black}}] (-3.04,0.98) -- (-9.72,2.6);
\draw [line width=.8pt,draw=black,postaction={on each segment={mid arrow=black}}] (-3.14,7) -- (-3.04,0.98);
\draw [line width=.8pt,draw=black,postaction={on each segment={mid arrow=black}}] (0.42,-2.04) -- (0.38,3.96);
\draw [line width=.8pt,draw=black,postaction={on each segment={mid arrow=black}}] (8.0001,2.64) -- (-3.14,7);
\draw [line width=.8pt,draw=black,postaction={on each segment={mid arrow=black}}] (0.38,3.96) -- (8.0001,2.64);
\draw [line width=.8pt,draw=black,postaction={on each segment={mid arrow=black}}] (-3.04,0.98) -- (8.0001,2.64);
\draw [line width=.8pt,draw=black,postaction={on each segment={mid arrow=black}}] (8.0001,2.64) -- (0.42,-2.04);
\draw [line width=.8pt,draw=black,postaction={on each segment={mid arrow=black}}] (0.42,-2.04) -- (-3.04,0.98);
\draw [line width=.8pt,draw=black,postaction={on each segment={mid arrow=black}}] (-3.14,7) -- (0.38,3.96);
\draw [fill=black] (-9.72,2.6) circle [radius=0.2];
\node at (-9.72,2.6) [anchor=east] {$a$};
\draw [fill=black] (-3.14,7) circle [radius=0.2];
\node at (-3.14,7) [anchor=south] {$b$};
\draw [fill=black] (0.42,-2.04) circle [radius=0.2];
\node at (0.42,-2.04) [anchor=north] {$c$};
\draw [fill=black] (0.38,3.96) circle [radius=0.2];
\node at (0.38,3.96) [anchor=south] {$d$};
\draw [fill=black] (-3.04,0.98) circle [radius=0.2];
\node at (-3.04,0.98) [anchor=north] {$e$};
\draw [fill=black] (8.0001,2.64) circle [radius=0.2];
\node at (8.0001,2.64) [anchor=west] {$f$};
\end{tikzpicture}~
{\small
\begin{tikzpicture}[scale=.47]
\draw[line width=.8pt,draw=black,postaction={on each segment={mid arrow=black}}] (-7.04,-0.01) -- (0,4.01);
\draw[line width=.8pt,draw=black,postaction={on each segment={mid arrow=black}}] (0.006,-4.014) -- (-7.04,-0.01);
\draw[line width=.8pt,draw=black,postaction={on each segment={mid arrow=black}}] (0,4.01) -- (0.006,-4.014);
\draw[line width=.8pt,draw=black,postaction={on each segment={mid arrow=black}}] (7.024,-0.01) -- (0,4.01);
\draw[line width=.8pt,draw=black,postaction={on each segment={mid arrow=black}}] (0.006,-4.014) -- (7.024,-0.01);
\draw [fill=white] (-7.04,-0.01) circle [radius=1.1];
\draw [fill=white] (0,4.01) circle [radius=1.1];
\draw [fill=white] (0.006,-4.014) circle [radius=1.1];
\draw [fill=white] (7.024,-0.01) circle [radius=1.1];
\node at (-7.04,-0.01) {$\P_0$};
\node at (0,4.01) {$\P_1$};
\node at (0.006,-4.014) {$\P_2$};
\node at (7.024,-0.01) {$\P_3$};
\node  at (-6.04,0.99) {$2$};
\node  at (-6.04,-1.01) {$2$};
\node  at (-7.04,-1.51) {--};
\node  at (0,5.51) {--};
\node  at (1.3,3.41) {$1$};
\node  at (-1.3,3.41) {$1$};
\node  at (0.3,2.51) {$2$};
\node  at (1.3,-3) {$1$};
\node  at (-1.3,-3) {$1$};
\node  at (0.006,-5.51) {--};
\node  at (0.3,-2.51) {$2$};
\node  at (7.024,-1.51) {--};
\node  at (5.624,0.9) {$2$};
\node  at (5.624,-0.7) {$2$};
\end{tikzpicture}
}
\caption{Directed graph $\G$ (from \cite[Example~5.4]{GCs}) of diameter $3$ and the intersection diagram of an equitable distance-faithful  partition $\Pi_a=\{\P_0=\{a\},\P_1=\{b,c\},\P_2=\{d,e\},\P_3=\{f\}\}$ of $\G$ (around vertex $a$). The adjacency matrix of this graph generates a commutative $3$-class association scheme. Note that $\G_1(a)=\P_1$, $\G_2(a)=\P_2$ and $\G_3(a)=\P_3$.}
\label{2f}
\end{center}
\end{figure}
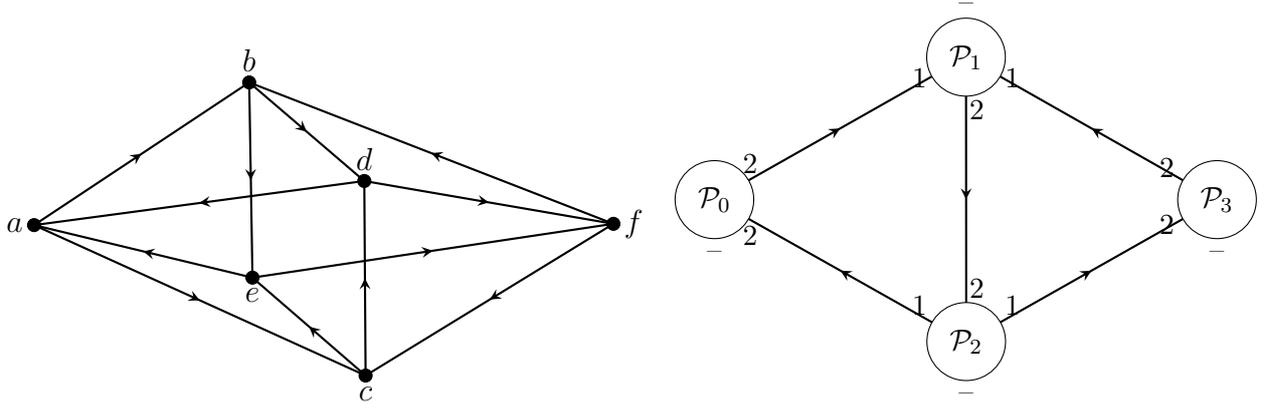

Let $\G = (X,\E)$ denote a directed graph. For any two vertices $x, y \in X$, a {\em directed walk} of length $h$ from $x$ to $y$ is a sequence $[x_0,x_1,x_2,\ldots,x_h]$ $(x_i\in X,\, 0\le i\le h)$ such that $x_0 = x$, $x_h = y$, and $x_i$ is adjacent to $x_{i+1}$ (i.e. $x_{i+1}\in\G^{\ra}_1(x_i)$) for $0\le i\le h-1$. We say that $\G$ is {\em strongly connected} if for any $x, y\in X$ there is a directed walk from $x$ to $y$. A {\em closed directed walk} is a directed walk from a vertex to itself. A {\em directed path} is a directed walk such that all vertices of the directed walk are distinct. A {\em cycle} is a closed directed path.

For any $x, y\in X$, the {\em distance} between $x$ and $y$, denoted by $\partial(x, y)$, is the length of a shortest directed path from $x$ to $y$. The {\em diameter} $D = D(\G)$ of a strongly connected directed graph $\G$ is defined to be 
$$
D = \max\{\partial(y,z)\,|\,y, z\in X\}.
$$
For a vertex $x\in X$ and any non-negative integer $i$ not exceeding $D$, let $\G^\ra_i(x)$ (or $\G_i(x)$) denote the subset of vertices in $X$ that are at distance $i$ from $x$, i.e.,
$$
\G^\ra_i(x)=\{z\in X\mid \partial(x,z)=i\}.
$$
We also define the set $\G^\la_i(x)$ as $\G^\la_i(x)=\{z\in X\mid \partial(z,x)=i\}$. Let $\G_{-1}(x) = \G_{D+1}(x) := \emptyset$. The elements of $\G_1(x)(=\G_1^{\ra}(x))$ are called {\em neighbors} of $x$. The {\it eccentricity} of $x$, denoted by $\varepsilon=\varepsilon(x)$, is the maximum distance between $x$ and any other vertex of $\G$. Note that the diameter of $\G$ equals $\max\{\varepsilon(x)\mid x\in X\}$.

All undirected graphs in this paper can be understood as directed graphs in which an undirected edge between two vertices $x$ and $y$ represents two arcs, an arc from $x$ to $y$, and an arc from $y$ to $x$. In diagrams instead of drawing two arcs we draw one undirected edge between vertices $x$ and $y$.  For a basic introduction to the theory of undirected graphs we refer to \cite[Section~2]{FMPS}. With the word {\em graph} we refer to a finite simple undirected graph.

We say that a graph $\G$ is {\em $N$-partite} if its set of vertices can be decomposed into $N$ disjoint sets such that no two vertices within the same set are adjacent. If $N=2$ such graphs are called {\em bipartite}. An {\em $N$-partite complete graph} $\G$ is $N$-partite graph for which there is an edge between every pair of vertices from different (disjoint) sets.


\subsection{Equitable and distance-faithful partition}

A {\it partition} of a (directed) graph $\G$ is a collection $\{\P_0, \P_1, \dots, \P_s\}$ of nonempty subsets of the vertex set $X$, such that $X=\ds\bigcup_{i=0}^s \P_i$ and $\P_i\cap\P_j=\emptyset$ for all $i,j$ $(0 \le i,j \le s,\, i\ne j)$. An {\it equitable partition} of a directed graph $\G$ is a partition $\{\P_0, \P_1, \dots, \P_s\}$ of its vertex set, such that for all integers $i,j$ $(0 \le i,j \le s)$ the following two conditions hold.
\begin{enumerate}[label=\rm(\roman*)]
\item The number $d^{\ra}_{ij}$ of neighbors which a vertex in the cell $\P_i$ has in the cell $\P_j$ is independent of the choice of the vertex in $\P_i$ (i.e., for every $y\in\P_i$ we have $|\G^{\ra}_1(y)\cap\P_j|=d^{\ra}_{ij}$).
\item The number $d^{\la}_{ij}$ of vertices from the cell $\P_j$ which are adjacent to a vertex in $\P_i$ is independent of the choice of the vertex in $\P_i$ (i.e., for every $y\in\P_i$ we have $|\sum_{z\in\P_j}|\G^{\ra}_1(z)\cap\{y\}|=d^{\la}_{ij}=|\G_1^\la(y)\cap\P_j|$).
\end{enumerate}
We call the numbers $d^{\ra}_{ij}$ and $d^{\la}_{ij}$ $(0\le i,j\le s)$ the {\em corresponding parameters}.

A {\it distance partition around} $x$ of a (directed) graph $\G$ with vertex set $X$ is a partition $\{\G_0(x)=\{x\},\G(x),\ldots,\G_{\varepsilon(x)}(x)\}$ of $X$ where $\varepsilon(x)$ is eccentricity of $x$. A {\it $x$-distance-faithful partition} $\{\P_0=\{x\},\P_1,\ldots,\P_s\}$ with $s\ge\varepsilon$ is a refinement of the distance partition around $x$ (here refinement means that some of $\G_i(x)$ can be equal to a union of some $\P_h$'s).

The {\it intersection diagram} of an equitable partition $\Pi$ of a graph $\G$ is a collection of circles indexed by the sets of $\Pi$ with lines (or directed edges) between some of them. If there is no line (directed edge) between $\P_i$ and $\P_j$, then it means that there is no (directed) edge $yz$ for any $y\in\P_i$ and $z\in\P_j$. If there is a line (directed edge) between $\P_i$ and $\P_j$, then a number on the line (from $\P_i$ to $\P_j$) near the circle $\P_i$ denotes the corresponding parameter $d^{\ra}_{ij}$. A number above or below a circle $\P_i$ denotes the corresponding parameter $d^{\ra}_{ii}(=d^{\la}_{ii})$. A similar explanation holds for the corresponding parameter $d^{\la}_{ij}$ (see Figures~\ref{2f} and \ref{2e} for an example).

We say that the combinatorial structure of the intersection diagram is {\em the same} around every vertex if for every vertex $x$ there exists an $x$-distance-faithful equitable partition with the same number of cells of same cardinality and (same) corresponding parameters do not depend on the choice of~$x$.

\begin{figure}[t!]
\begin{center}
\begin{tikzpicture}[scale=.43]
\draw [line width=.8pt] (-7.04,-0.01)-- (0,2.53);
\draw [line width=.8pt] (0,2.53)-- (-0.02,-2.51);
\draw [line width=.8pt] (-0.02,-2.51)-- (6.98,-0.99);
\draw [line width=.8pt] (6.98,-0.99)-- (6.98,1.01);
\draw [line width=.8pt] (6.98,1.01)-- (1,-5.51);
\draw [line width=.8pt] (1,-5.51)-- (1.02,5.51);
\draw [line width=.8pt] (1.02,5.51)-- (-7.04,-0.01);
\draw [line width=.8pt] (1,-5.51)-- (-7.04,-0.01);
\draw [line width=.8pt] (-7.04,-0.01)-- (-0.02,-2.51);
\draw [line width=.8pt] (-0.02,-2.51)-- (6.98,1.01);
\draw [line width=.8pt] (6.98,1.01)-- (1.02,5.51);
\draw [line width=.8pt] (1.02,5.51)-- (0,2.53);
\draw [line width=.8pt] (0,2.53)-- (6.98,-0.99);
\draw [line width=.8pt] (6.98,-0.99)-- (1,-5.51);
\draw [fill=black] (0,2.53) circle [radius=0.2];
\draw [fill=black] (-0.02,-2.51) circle [radius=0.2];
\draw [fill=black] (6.98,-0.99) circle [radius=0.2];
\draw [fill=black] (6.98,1.01) circle [radius=0.2];
\draw [fill=black] (1,-5.51) circle [radius=0.2];
\draw [fill=black] (1.02,5.51) circle [radius=0.2];
\draw [fill=black] (-7.04,-0.01) circle [radius=0.2];
{\tiny
\node at (-7.54,-0.01) {$0$};
\node at (0.5,2.53) {$1$};
\node at (-0.07,-3.01) {$2$};
\node at (7.48,-0.99) {$3$};
\node at (7.48,1.01) {$4$};
\node at (1.5,-5.51) {$5$};
\node at (1.52,5.51) {$6$};
}
\end{tikzpicture}\qquad\qquad
{\small
\begin{tikzpicture}[scale=.47]
\draw[line width=.8pt] (-7.04,-0.01) -- (0,4.01);
\draw[line width=.8pt] (-7.04,-0.01) -- (0.006,-4.014);
\draw[line width=.8pt] (0,4.01) -- (0.006,-4.014);
\draw[line width=.8pt] (0,4.01) -- (7.024,-0.01);
\draw[line width=.8pt] (0.006,-4.014) -- (7.024,-0.01);
\draw [fill=white] (-7.04,-0.01) circle [radius=1.1];
\draw [fill=white] (0,4.01) circle [radius=1.1];
\draw [fill=white] (0.006,-4.014) circle [radius=1.1];
\draw [fill=white] (7.024,-0.01) circle [radius=1.1];
\node at (-7.04,-0.01) {$\P_0$};
\node at (0,4.01) {$\P_1$};
\node at (0.006,-4.014) {$\P_2$};
\node at (7.024,-0.01) {$\P_3$};
\node  at (-6.04,0.99) {$2$};
\node  at (-6.04,-1.01) {$2$};
\node  at (-7.04,-1.51) {--};
\node  at (0,5.51) {$1$};
\node  at (1.3,3.41) {$1$};
\node  at (-1.3,3.41) {$1$};
\node  at (0.3,2.51) {$1$};
\node  at (1.3,-3) {$2$};
\node  at (-1.3,-3) {$1$};
\node  at (0.006,-5.51) {--};
\node  at (0.3,-2.51) {$1$};
\node  at (7.024,-1.51) {$1$};
\node  at (5.624,0.9) {$1$};
\node  at (5.624,-0.7) {$2$};
\end{tikzpicture}
}
\caption{Undirected graph $\G=\mbox{Cay}(\ZZ_{7};\{1,2\})$ of diameter $2$ and the intersection diagram of an equitable distance-faithful partition of $\G$ (around vertex $0$). The adjacency matrix of this graph generates a symmetric $3$-class association scheme.}
\label{2e}
\end{center}
\end{figure}
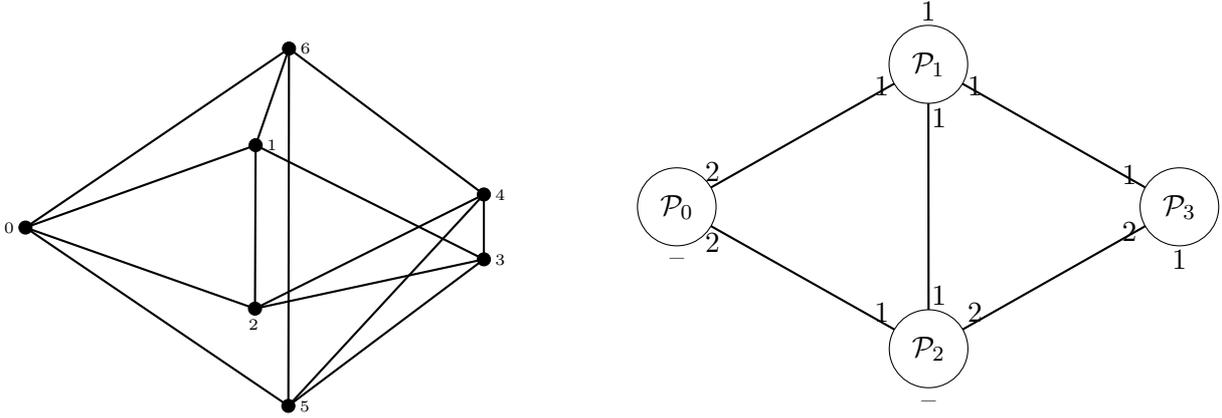

\subsection{Elementary algebraic graph theory}
\label{ga}

In this section, we recall some definitions and basic concepts from algebraic graph theory.

The {\em adjacency matrix} $A\in\Mat_X(\CC)$ of a directed graph $\G$ (with  vertex set $X$) is indexed by the vertices from $X$, and is defined in the following way
$$
\mbox{$(A)_{yz} =$ the number of arcs from $y$ to $z$}\qquad (y,z\in X)
$$
(note that $(A)_{yz}\ge 0$). Moreover, if we allow loops, the diagonal entries of $A$ can be different from zero. If there is at most one arc between pairs of vertices (i.e. the adjacency matrix $A$ is a $01$-matrix), the $yz$-entry of the power $A^\ell$ $(\ell\in\NN)$ corresponds to the number of $\ell$-walks from the vertex $y$ to the vertex $z$ in $\G$ (the proof is similar to \cite[Lemma~3.1]{SPm}).

\begin{lemma}
\label{hB}
Let $\G$ denote a simple strongly connected digraph with vertex set $X$, diameter $D$ and adjacency matrix $A$. The number of walks of length $\ell\in\NN$ in $\G$ from $x$ to $y$ is equal to $(x,y)$-entry of the matrix $A^\ell$.
\end{lemma}

The distance-$i$ matrix $A_i$ of a digraph $\G$ with diameter $D$ and vertex set $X$ is defined by
$$
(A_i)_{zy}=\left\{\begin{matrix}
1&\mbox{ if $\partial(z,y)=i$},\\
0&\mbox{ otherwise.~~~~}
\end{matrix}\right. \qquad(z,y\in X,~0\le i\le D).
$$
In particular, $A_0=I$ and $A_1=A$. A matrix $A\in\Mat_X(\CC)$ is said to be a {\em reducible} when there exists a permutation matrix $P$ such that $P^\top A P=\left(
\begin{matrix}
X&Y\\ \O&Z
\end{matrix}
\right)$, where $X$ and $Z$ are both square, and $\O$ is a zero matrix. Otherwise, $A$ is said to be {\em irreducible}.

\begin{lemma}
\label{gb}
A directed graph $\G$ with adjacency matrix $A$ is strongly connected if and only if $A$ is an irreducible matrix.
\end{lemma}

\begin{proof}
Routine. (See, for example, \cite[Section~8.3]{MCm}.)
\end{proof}


\begin{theorem}[Perron--Frobenius Theorem]
\label{gc}
Let $\G=\G(A)$ denote a directed strongly connected graph with spectrum $\spec(\G)$. If $\theta=\max_{\lambda\in\spec(\G)}|\lambda|$, then the following hold.
\begin{enumerate}[label=\rm(\roman*)]
\item $\theta\in\spec(\G)$.
\item The algebraic multiplicity of $\theta$ is equal to $1$.
\item There exists an eigenvector $\nnu$ with all positive entries, such that $A\nnu = \theta \nnu$.
\end{enumerate}
Sometimes it is useful to normalize a vector $\nnu$ from {\rm (iii)} in such a way that the smallest entry is equal to $1$. Such a vector $\nnu$ is called a Perron--Frobenius eigenvector.
\end{theorem}

\begin{proof}
Routine using Lemma~\ref{gb}. (See, for example, \cite[Section~8.3]{MCm}.)
\end{proof}

\smallskip
A matrix $A\in\Mat_X(\CC)$ is called {\em normal} if it commutes with its adjoint, i.e. if $A\ol{A}^\top=\ol{A}^\top A$.

\begin{theorem}
\label{gd}
Let $A\in\Mat_X(\CC)$ denote a matrix over $\CC$, with rows and columns indexed by $X$. Then, the following are equivalent.
\begin{enumerate}[label=\rm(\roman*)]
\item $A$ is normal.
\item $\CC^{|X|}$ has an orthonormal basis consisting of eigenvectors of $A$.
\item $A$ is a diagonalizable matrix.
\item The algebraic multiplicity of $\lambda$ is equal to the geometric multiplicity of $\lambda$, for every eigenvalue $\lambda$ of $A$.
\end{enumerate}
\end{theorem}

\begin{proof}
Routine. (See, for example, \cite[Chapter~7]{AS}.)
\end{proof}

Two matrices $A,B\in\Mat_X(\CC)$ are said to be {\em simultaneously diagonalizable} if there is a nonsingular $S\in\Mat_X(\CC)$ such that $S^{-1}AS$ and $S^{-1}BS$ are both diagonal.

\begin{lemma}[{\cite[Theorem 1.3.12]{HJ}}]
\label{ge}
Two diagonalizable matrices are simultaneously diagonalizable if and only if they commute.
\end{lemma}

\begin{theorem}
\label{gf}
Let $\M$ denote a space of commutative normal matrices. Then, there exists a unitary matrix $U\in\Mat_X(\CC)$ which diagonalizes $\M$.
\end{theorem}

\begin{proof}
Immediate from Theorem~\ref{gd}, Lemma~\ref{ge} and \cite[Subsection~1.3]{HJ}.
\end{proof}

Let $\G$ denote a regular graph with vertex set $X$ and $\circ$ denote the elementwise--Hadamard product of matrices. Let us call two $01$-matrices $B$, $C$ {\em disjoint} if $B\circ C=0$. For the moment, let $\B$ denote some algebra of $|X|\times |X|$ matrices. A basis $\{B_0,B_1,\ldots,B_d\}$ of $\B$ is called a {\em standard basis} of $\B$ if and only if the $B_i$'s are mutually disjoint $01$-matrices which satisfy the following properties: {\rm (i)}~the sum of some of these matrices gives $I$; {\rm(ii)}~the sum of all of these matrices gives the all-$1$ matrix $J$; {\rm(iii)}~for each $i\in \{0,\ldots,d\}$, the conjugate transpose of $B_i$ belongs to $\{B_0,B_1,\ldots,B_d\}$; and {\rm(iv)}~the vector space spanned by $\{B_0,B_1,\ldots,B_d\}$ is closed under both ordinary and elementwise--Hadamard multiplication.


\subsection{Commutative association scheme}\label{COM}

Let $X$ denote a finite set and $\Mat_X(\CC)$ the set of complex matrices with rows and columns indexed by $X$. Let $\R=\{R_0,R_1,\ldots,R_d\}$ denote a set of cardinality $d+1$ of nonempty subsets of $X\times X$. The elements of the set $\R$ are called {\em relations} (or {\em classes}) on $X$. For each integer $i$ $(0\le i\le d)$, let $A_i\in\Mat_X(\CC)$ denote the adjacency matrix of the graph $(X,R_i)$ (directed, in general). The pair ${\XXi}=(X,\R)$ is a {\em commutative $d$-class association scheme} (or a {\em $d$-class scheme} for short) if
\begin{enumerate}[leftmargin=1.7cm,label=\rm(AS\arabic*)]
\item $A_0=I$, the identity matrix.\label{ce}
\item $\ds{\sum_{i=0}^d A_i=J}$, the all-ones matrix.\label{cg}
\item ${A_i}^\top\in\{A_0,A_1,\ldots,A_d\}$ for $0\le i\le d$.
\item $A_iA_j$ is a linear combination of $A_0,A_1,\ldots,A_d$ for $0\le i,j\le d$ (i.e., for every $i,j$ $(0\le i,j\le d)$ there exist {\em intersection numbers} $p^h_{ij}$, $0\le h\le d$, such that $A_iA_j=\sum_{h=0}^d p^h_{ij} A_h$).\label{cb}
\item $A_iA_j=A_jA_i$ for every $i,j$ $(0\le i,j\le d)$ (i.e., for the intersection numbers $p^h_{ij}$, $0\le i,j,h\le d$, from \ref{cb} we have that $p^h_{ij}=p^h_{ji}$).\label{cf}
\end{enumerate}

By \ref{ce}--\ref{cf} the vector space $\M=\Span\{A_0,A_1,\ldots,A_d\}$ is a commutative algebra; we call it the {\em Bose--Mesner algebra} of $\XXi$. The set of $(0,1)$-matrices $\{A_0,A_1,\ldots,A_d\}$ is linearly independent by \ref{cg} and thus forms a basis of $\M$. We say that $\XXi$ is {\em symmetric} if the $A_i$'s are symmetric matrices.

For the moment, pick $h$ $(0\le h\le d)$ and let $x,y\in X$ denote two vertices such that $(A_h)_{xy}=1$. By \ref{cg} and \ref{cb}, $(A_iA_j)_{xy}=p^h_{ij}$ $(0\le i,j\le d)$. On the other hand
\begin{align*}
(A_iA_j)_{xy}&=\sum_{z\in X} (A_i)_{xz} (A_j)_{zy}\\
&=|\{z\in X\mid (A_i)_{xz}=1 \mbox{ and } (A_j)_{zy}=1\}|\\
&=|\{z\in X\mid (x,z)\in R_i \mbox{ and } (z,y)\in R_j\}|,
\end{align*}
which yields $p^h_{ij}=|\{z\in X\mid (x,z)\in R_i \mbox{ and } (z,y)\in R_j\}|$. This suggests an equivalent {\em combinatorial} definition of a commutative association scheme (the following axioms are the combinatorial analogs of those given in \ref{ce}--\ref{cf}):
\begin{enumerate}[leftmargin=1.7cm,label=\rm(AS\arabic*')]
\item $R_0 = \{(x, x) \mid x \in X\}$ (that is, $R_0$ is the \emph{diagonal relation}).
\item $\{R_i\}_{i=0}^d$ is a partition of the Cartesian product $X\times X$.
\item Relation $R^\top_j = \{(y, x)\mid (x, y)\in R_j\}$ is in $\{R_i\}_{0\le i\le d}$, for each $j\in\{0,\ldots , d\}$ (that is, $\{R_i\}_{i=0}^d$ is closed under taking the {\em transpose relation $^\top$}).
\item For each triple $i,j,h$ $(0\le i,j,h\le d)$, and $(x, y) \in R_h$, a scalar
\begin{equation}
\label{cv}
|\{z\in X \mid (x,z)\in R_i \hbox{ and } (z,y)\in R_j\}|
\end{equation}
does not depend on the choice of the pair $(x, y) \in R_h$. The scalars obtained from line \eqref{cv} we denote by $p^h_{ij}$ and call the {\em intersection numbers} of $\XXi$.
\item For each triple $i,j,h$ $(0\le i,j,h\le d)$, $p^h_{ij}=p^h_{ji}$.
\end{enumerate}
Note that association scheme is \emph{symmetric} if $R_i=R_i^\top$, for each $i$ $(0\le i\le d)$. Immediately from the combinatorial definition, for example, we can get some properties on the intersection numbers that we will use later: in particular, pick $j$ $(0\le j\le d)$, let $(x,y)\in R_j$, and note that
\begin{align*}
\sum_{\ell=0}^d p^j_{k\ell} 
&= \sum_{\ell=0}^d |\{z\in X\mid (x,z)\in R_k \mbox{ and } (z,y)\in R_\ell\}|\\
&=|\{z\in X\mid (x,z)\in R_k\}|\\
&= p^0_{kk}.
\end{align*}
We abbreviate $n_k:=p^0_{kk}$ $(0\le k\le d)$. The number $n_k$ is the so-called \emph{valency} of the relation $R_k$. For any $w\in X$, the comments from above imply
\begin{equation}
\label{ck}
(A_k\jj)_w=|\{z\in X\mid (w,z)\in R_k\}|=p^0_{kk}=n_k=\sum_{\ell=0}^d p^j_{k\ell}\qquad(0\le k\le d).
\end{equation}
Equation \eqref{ck} also yields that all-$1$ vector $\jj$ is an eigenvector of $A_k$ $(0\le k\le d)$ that corresponds to the eigenvalue $n_k$.

\begin{lemma}
\label{cc}
With respect to the above notation, the Bose--Mesner algebra $\M$ of a commutative $d$-class association scheme $\XXi$ is a space of commutative normal matrices.
\end{lemma}

\begin{proof}
From the definition of $\XXi$, $A_i\overline{A_i}^\top=\overline{A_i}^\top A_i$ $(0\le i\le d)$, and the result follows.
\end{proof}

\medskip
Let $V=\CC^{|X|}$ denote the set of complex column vectors with coordinates indexed by $X$, and observe that $\Mat_X(\CC)$ acts on $V$ from the left. We endow $V$ with the Hermitian inner product $\langle \cdot , \cdot \rangle$ that satisfies $\langle u,v \rangle = \overline{u}^\top v$ for $u,v \in V$, where ``$\top$" denotes transpose and ``$\overline{\phantom{v}}$" denotes complex conjugation. By Lemma~\ref{cc} and Theorem~\ref{gf}, the matrices of $\M$ are simultaneously diagonalizable by a unitary matrix. This yields that there is a unitary matrix $U\in\Mat_X(\CC)$ such that $\overline{U}^\top\M U$ consists of diagonal matrices only. Each column $\uu$ of $U$ is a common eigenvector of $A_0$, $A_1,\ldots,A_d$. We can permute columns of matrix $U$ and collect them together in a partition of $r+1$ different matrices $U_i$
$$
U=\left(
\begin{matrix}
|&|&~&|\\
U_0&U_1&\cdots&U_r\\
|&|&~&|
\end{matrix}
\right)
$$
in such a way that the following holds 
\begin{enumerate}[label=$\bullet$]
\item there exists a complex scalar $p_h(i)$ $(0\le i\le r, 0\le h\le d)$ such that for each column $\uu$ of $U_i$ we have $A_h\uu=p_h(i)\uu$.
\end{enumerate}
Let $V_i$ denote the vector space spanned by the columns of $U_i$ $(0\le i\le r)$, i.e. $V_i=\im U_i$ $(0\le i\le r)$. The space $V_i$ is a {\em common eigenspace} of $A_0,A_1,\ldots,A_d$. The set of the common eigenspaces $\{V_0,V_1,\ldots,V_r\}$ is {\em maximal} if for any $i\ne j$ $(0\le i,j\le r)$ there exists $A_h$ such that the eigenvalue of $A_h$ on $V_i$ is different from that on $V_j$.

Using the above notation it is not hard (but also not so easy) to prove that
$
d=r
$
and that the set $\{E_0,E_1,\ldots,E_d\}$ of matrices $E_i$'s, which are defined as follows
\begin{equation}
\label{ci}
E_i=U_i\overline{U_i}^\top\quad(0\le i\le d),
\end{equation} 
is an another basis of $\M$. Moreover, we have
(ei) $E_iE_j=\delta_{ij}E_i$ $(0\le i,j\le d)$; 
(eii)\label{cl} $\ds{\sum_{i=0}^d E_i=I}$, the identity matrix;
(eiii) There exists a complex scalar $p_h(i)$ $(0\le i, h\le d)$ such that $A_hE_i=p_h(i)E_i$ (moreover, $p_h(i)$ is the eigenvalue of $A_h$ on the eigenspace $V_i$); 
(eiv) $A_h\in\Span\{E_0,E_1,\ldots,E_d\}$ $(0\le h\le d)$; 
(ev) $\ol{E_i}^\top=E_i$ $(0\le i\le d)$;
(evi) the idempotents $E_i$ are the orthogonal projectors of $V$ onto the spaces $V_i$.

The change-of-basis matrices $P$ and $Q$ are defined by
$$
A_i=\sum_{h=0}^d (P)_{hi} E_h,\qquad
E_i=\frac{1}{|X|}\sum_{h=0}^d (Q)_{hi}A_h.
$$
We shall refer to $P$ and $Q$ as the {\em first} and {\em second eigenmatrix} of the association scheme, respectively. Moreover, we set
\begin{equation}
\label{cs}
P=\left(
\begin{matrix}
p_0(0)&p_1(0)&p_2(0)&\cdots&p_d(0)\\
p_0(1)&p_1(1)&p_2(1)&\cdots&p_d(1)\\
p_0(2)&p_1(2)&p_2(2)&\cdots&p_d(2)\\
\vdots&\vdots&\vdots& ~ &\vdots\\
p_0(d)&p_1(d)&p_2(d)&\cdots&p_d(d)\\
\end{matrix}
\right)=
\left(
\begin{matrix}
-\!\!\!-\!\!\!-\!\!- & (P)_{0*} & -\!\!\!-\!\!\!-\!\!-\\
-\!\!\!-\!\!\!-\!\!- & (P)_{1*} & -\!\!\!-\!\!\!-\!\!-\\
-\!\!\!-\!\!\!-\!\!- & (P)_{2*} & -\!\!\!-\!\!\!-\!\!-\\
~& \vdots & ~\\
-\!\!\!-\!\!\!-\!\!- & (P)_{d*} & -\!\!\!-\!\!\!-\!\!-\\
\end{matrix}
\right).
\end{equation}

\begin{lemma}[{\cite{BI}}]
\label{ct}
With respect to the above notation, let $\{V_i\}_{i=0}^d$ denote a set of maximal common eigenspaces of $\{A_h\}_{h=0}^d$ and let $\jj$ denote the all-$1$ vector. The first eigenmatrix $P$ has the following form
$$
P = \begin{blockarray}{cccccc}
~ & R_0 & R_1 & R_2 & \cdots & R_d \\
\begin{block}{c(ccccc)}
V_0&1&n_1&n_2&\cdots&n_d\\
V_1&1&p_1(1)&p_2(1)&\cdots&p_d(1)\\
V_2&1&p_1(2)&p_2(2)&\cdots&p_d(2)\\
\vdots&\vdots&\vdots&\vdots& ~ &\vdots\\
V_d&1&p_1(d)&p_2(d)&\cdots&p_d(d)\\
\end{block}
\end{blockarray}
$$
where $n_i$ are positive integers. Moreover, $A_i\jj=n_i\jj$ $(0\le i\le d)$ and for any $i$ $(0\le i\le d)$, the scalars $n_i,p_i(1),\ldots,p_i(d)$ are the eigenvalues (not necessarily pairwise distinct) of $A_i$ on $V_0,V_1,\ldots,V_d$, respectively.
\end{lemma}

The matrix $P$ is also called the {\em character table} of an association scheme, and in fact can be viewed as a natural generalization of the character table of a finite group (see, for example, \cite{BHS,CX,KH,XB}).

\begin{lemma}[{\cite[Proposition~3.4]{BI}}]
\label{cu}
With respect to the above notation, let $\XXi=(X,\R)$ denote a $d$-class association scheme and $V_i=E_iV$ $(0\le i\le d)$. Then the second eigenmatrix $Q$ has the following form
$$
Q = \begin{blockarray}{cccccc}
~ & V_0 & V_1 & V_2 & \cdots & V_d \\
\begin{block}{c(ccccc)}
R_0&1&m_1&m_2&\cdots&m_d\\
R_1&1&q_1(1)&q_2(1)&\cdots&q_d(1)\\
R_2&1&q_1(2)&q_2(2)&\cdots&q_d(2)\\
\vdots&\vdots&\vdots&\vdots& ~ &\vdots\\
R_d&1&q_1(d)&q_2(d)&\cdots&q_d(d)\\
\end{block}
\end{blockarray}
$$
where $m_i=\dim(V_i)$.
\end{lemma}

\begin{lemma}
\label{cp}
Let $P$ and $Q$ denote the first and second eigenmatrices of an association scheme $\XXi=(X,\R)$, respectively. Then,
$$
PQ=QP=|X|I.
$$
\end{lemma}

\begin{proof}
Immediate from the definitions of $P$ and $Q$.
\end{proof}

\begin{corollary}
\label{cq}
Let $P$ denote the first eigenmatrix of an association scheme $\XXi=(X,\R)$. Then,
$$
P\jj=\left(\begin{matrix}|X|\\0\\\vdots\\0\end{matrix}\right),
$$
i.e., for every $i$ $(i\ne 0, 1\le i\le d)$ the sum of the entries of the $V_i$ row in $P$ is equal to $0$.
\end{corollary}

\begin{proof}
Immediate from Lemmas~\ref{cu} and \ref{cp}.
\end{proof}

An association scheme $(X,\S)$ on the same vertex set $X$ is called a {\em fusion} of $(X,\R)$ if each $S_i\in\S$ is the union of some of the $R_i$. Note that $R_0\in\S$. As an extreme case, we call $(X,\R)$ {\em amorphous} (or {\em amorphic}) if every ``merging'' operation on $\{R_0,R_1,\ldots,R_d\}$ yields a fusion (see \cite{vDM} for survey on this topic). As the adjacency matrices of a fusion of $(X,\R)$ must be $01$-linear combinations of the $A_i$, it is theoretically possible to find all fusions of $(X,\R)$ from the first eigenmatrix $P$. This is accomplished using the {\em Bannai--Muzychuk Criterion} {\cite[Lemma~1]{BEs}} (see e.g., \cite{BEs,TH,FT,ITd} for examples of explicit constructions of fusions, including $3$-class association schemes, that use this criterion). It is worth underlining that in Section~\ref{ra} we consider unions of two or more relations of a $3$-class association scheme as candidates to generate the whole scheme, and Proposition~\ref{do} provides a criterion, based on the entries of the first eigenmatrix $P$, to select \emph{generators} (recall that a matrix $A$ {\em generates} $\M$ if every element in $\M$ can be written as a polynomial in $A$). Even if the premises are similar, this is quite different from the Bannai--Muzychuk criterion: Bannai and Muzychuk look for unions of relations of a $d$-class scheme that can give rise to a $d'$-class subscheme with $d'<d$; while in our case we are interested in finding a union of relations of a scheme, i.e., a sum of adjacency matrices, say $A$, such that any other matrix of the Bose--Mesner of the scheme can be expressed as a polynomial in $A$. 




\subsection{On $\boldsymbol{2}$-class association schemes: strongly-regular graphs}
\label{SRG}

In order to better understand some arguments and results in Section \ref{ra}, it is convenient to recall what is widely known about strongly-regular graphs. In particular, we define a strongly-regular graph by using the language of association schemes, and deduce its particular combinatorial properties from this definition. Furthermore, we reprove a well-known result about strongly-regular graphs, namely, that the adjacency matrix $A$ of a connected strongly-regular graph generates its corresponding association scheme (see Proposition \ref{gen}). We refer the reader to \cite{BCN, BL, CL, Sei}  for further details on the general theory of strongly-regular graphs, and we also point out the more recent \cite{BH, BVM, GR}. Our main source for what follows is \cite{Cam}.

Let $\Gamma=(X,R)$ denote a graph with vertex set $X$ and edge set $R$. Define $\overline{R}=\{(x,y) \in X\times X \mid (x,y)\notin R\}$ and  $R_0=\{(x,x): x\in X\}$.
The graph $\Gamma$ is said to be \emph{strongly-regular} if exactly one of the following two properties holds:

\begin{enumerate}[label=\rm(\roman*)]
\item $(X, \{R_0,R\})$ is the trivial association scheme (i.e., $\G$ is a complete graph),
\item $(X, \{R_0,R,\overline{R}\})$ is a (symmetric) $2$-class association scheme.
\end{enumerate}


A \emph{clique} $C$ of an undirected graph $\Gamma$ is an induced subgraph of $\Gamma$ such that every two distinct vertices of $C$ are adjacent (i.e., a clique of $\Gamma$ is a complete subgraph of $\Gamma$). The number of vertices of $C$ is called the \emph{size} of the clique $C$.

\begin{theorem}[{\cite[Theorem 3.11]{Cam}}]
A disconnected strongly-regular graph is a disjoint union of cliques of the same size. Conversely, if a graph is a disjoint union of $t>1$ $(t \in \NN)$ cliques of the same size, then it is a disconnected strongly-regular graph.
\end{theorem}
\begin{proposition}[{\cite[page 2]{BVM}}]\label{low}
Let $\Gamma$ denote a strongly-regular graph. If $\Gamma$ is a disjoint union of cliques, then $-1$ is an eigenvalue for $\Gamma$, and vice-versa.
\end{proposition}

Throughout this section, we assume that $\Gamma=(X,R_1)$ is a connected strongly-regular graph for which the corresponding association scheme $\XXi=(X,\{R_0,R_1,R_2\})$ has $2$ classes. Directly from our definition, it is not hard to obtain the following combinatorial properties: the graph $\Gamma$ has valency $p^0_{11}$, and the number of common neighbors of two vertices $x, y$ of $\Gamma$ is $p^1_{11}$ or $p^2_{11}$, depending on whether $x$ and $y$ are adjacent or not (see Proposition \ref{gen} below for more details). Moreover, $\Gamma$ has diameter $2$. By convention, the intersection numbers $p^0_{11}$, $p^1_{11}$ and $p^2_{11}$ of the scheme $\XXi$ are denoted by $k$, $\lambda$ and $\mu$, respectively. Furthermore, we use $k_2$ to refer to $p^0_{22}$, i.e., the number of vertices that are not adjacent to a given one; thus, the number of vertices $|X|$ is equal to
\begin{equation}\label{w}
w=1+k+k_2. 
\end{equation}
Counting in two different ways the edges between vertices which are adjacent and nonadjacent to a fixed $x\in X$, we get 
the well-known identity
\begin{equation}\label{count}
k(k-\lambda-1)=k_2\mu.
\end{equation}
The \emph{$i$th intersection matrix} $L_i$ of a $d$-class association scheme is defined to be a $(d+1)$-matrix whose generic entry is $(L_i)_{j,k}=p^k_{ij}$, for $i,j,k \in \{0,\ldots,d\}$. Following the monumental thesis of {\sc Delsarte} \cite{DP}, after {\emph{diagonalizing} both sides of the equation in \ref{cb}, we deduce that $PL_iP^{-1}={\rm diag}(p_i(0), p_i(1), \ldots, p_i(d))$ (see \cite[page 13] {DP}). Consequently, the matrices $A_i$ and $L_i$ have the same eigenvalues (but with different multiplicities), and it follows that the map $A_i \rightarrow L_i$ defines an isomorphism between the Bose--Mesner algebra of the scheme and the algebra generated by the $L_i$'s. Thus, following \cite[pages 76, 77]{Cam}, the eigenvalues of $\Gamma$ are the zeros of the minimal polynomial of the matrix $L_1$ of the $2$-class association scheme $\XXi$. In particular, the valency $k$ is an eigenvalue with multiplicity $1$, and the other eigenvalues are 
\begin{equation}\label{rs}
r,s=\frac{(\lambda-\mu)\pm\sqrt{(\lambda-\mu)^2+4(k-\mu)}}{2}.
\end{equation}
The first eigenmatrix $P$ of the $2$-class scheme corresponding to $\Gamma$ is
\[P= \,
\begin{blockarray}{ccc}
R_0 & R_1& R_2\\
\begin{block}{(ccc)}
1& k& -1 - k + w\\
1& r& -1 - r\\
1&s&-1 - s\\
\end{block}
\end{blockarray}\ , 
\]
The multiplicities $f,g$ of eigenvalues $r$ and $s$ can be computed by solving the following equations:
\begin{equation}\label{first}
f+g=k+k_2 
\end{equation}
and
\begin{equation}\label{second}
k+fr+gs=\rm{trace}(A_1)=0
\end{equation}
(see, for example, \cite[page~77]{Cam}). Equation (\ref{first}) is due to  $PQ=wI$ (see Lemma \ref{cp}), which yields that the sum of the multiplicities of the scheme $\XXi$ is equal to $w$. Equation (\ref{second}) deserves several considerations. First, it is known that the trace of a matrix is invariant under diagonalization \cite[Chapter 1]{HJ}, and $A_1$ ($A_2$ as well) is diagonalizable (see Subsection \ref{COM}). This implies that  \rm{trace}($A_1$) is equal to the sum of the eigenvalues of $\Gamma$, counted with their multiplicities. On the other hand, since $\sum^2_{i=0}A_i=J$ (see axiom \ref{ce}), the diagonal entries of $A_1$ (those of $A_2$ as well) are zero. Now, equation (\ref{second}) follows.

\begin{remark}
\label{r_0}
{\rm{We show that $k=\mu$ if and only if $r=0$. First assume that $k=\mu$. Equation~\eqref{rs} yields $r=\frac{\lambda-\mu+\sqrt{(\lambda-k)^2}}{2}=\frac{\lambda-\mu+|\lambda-k|}{2}$, and with it $r=0$. Now assume that $r=0$. From equation~\eqref{rs} we now have $(\lambda-\mu)+\sqrt{(\lambda-\mu)^2+4(k-\mu)}=0$ which implies $k=\mu$. The claim follows.

Furthermore, from the first eigenmatrix $P$} it follows that $\Gamma=(X,R_1)$ has eigenvalue $r=0$ if and only if $-1$ is an eigenvalue for $\overline{\Gamma}=(X,R_2)$, i.e., $\overline{\Gamma}$ (which is also strongly-regular) is a disjoint union of cliques (by Proposition \ref{low}). Since $\overline{\Gamma}$ is disjoint union of cliques, $\Gamma$ is a complete multipartite graph (by construction). Now, for the case $k=\mu$, using \eqref{w}, \eqref{count}, and \eqref{rs}, we find 
\begin{equation}\label{r_zero}
w=2 k - \lambda, \ \ \ \ k_2=k - \lambda - 1, \ \ \ \ r=0, \ \ \ \ s = \lambda - k ~(<-1).
\end{equation}}
\end{remark}

In Remark~\ref{r_0} we considered the case $k=\mu$. For the rest of the current section assume $k\ne\mu$. By combining (\ref{rs}), (\ref{first}) and (\ref{second}), we get

\begin{equation}\label{multi}
f,g=\frac{1}{2}\left((k+k_2)\pm\frac{(k+k_2)(\mu-\lambda)-2k}{\sqrt{(\lambda-\mu)^2+4(k-\mu)}}\right)
\end{equation}

These numbers must be positive integers. We distinguish two cases in (\ref{multi}), so yielding two classes of strongly-regular graphs.

{\sc Case 1:}\label{case1} $(k+k_2)(\mu-\lambda)-2k=0$. Here, we find $k_2=k$ (since $k+k_2>k$ and $k+k_2$ divides $2k$). It follows that $\lambda=\mu-1$ and $k=f=g$. Moreover, \eqref{count} yields $k=2\mu$. Also, since $w=1+2k=1+4\mu$, we get
\begin{equation}\label{rs0}
r,s=\frac{-1\pm\sqrt{w}}{2},
\end{equation} 
with $r>0$ and $s<0$. Strongly-regular graphs with these parameters are known as \emph{conference graphs}. See \cite{BCN, Jon, WSW} or \cite[page~77]{Cam} for more details.

\medskip

{\sc Case 2:}\label{case2} $(k+k_2)(\mu-\lambda)-2k\neq0$. Now, following explanations in \cite[page~77]{Cam}, we deduce that $r$ and $s$ are integers, with $r>0$ and $s<0$ still holding, and use the eigenvalues $k,r,s$ to write the parameters of $\Gamma$: 

\begin{equation*}
\lambda=k+r+s+rs, \ \ \ \ \mu=k+rs, \ \ \ \ k_2=-\frac{k(r+1)(s+1)}{k+rs},
\end{equation*}
from which
\begin{equation}\label{para}
w=\frac{(k - r) (k - s)}{k + r s}.
\end{equation}

Therefore, with reference to \eqref{rs}, for a connected strongly-regular, if $k\ne\mu$ we can always assume that
\begin{equation}\label{rs1}
r>0 \hbox{ \, and \, } s<0 \hbox{ \, ($s\neq - 1$)}.
\end{equation}


We conclude the subsection by providing a well-known result, which is re-proved here, to prepare the reader for what is the spirit that runs through the upcoming sections.

\begin{proposition}
\label{gen}
Let $\Gamma$ denote a connected strongly-regular graph. Then, the following hold.
\begin{enumerate}[label=\rm(\roman*)]
\item $\Gamma$ is regular with valency $k$. Moreover, there exists a positive integer $\lambda$ such that any two adjacent vertices have $\lambda$ common neighbors. In addition, if $\Gamma$ is not a complete graph, there exists a positive integer $\mu$ such that the number of common neighbors of any two nonadjacent vertices is equal to $\mu$.
\item  If $\M$ is the Bose--Mesner algebra of the association scheme of $\Gamma$ and $A$ is the adjacency matrix of $\Gamma$, then $\M=(\langle A\rangle, +, \cdot)$.
\end{enumerate}

\end{proposition}
\begin{proof}
(i) Let $\Gamma=(X,R)$. According to our definition, there is a 1 or 2-class scheme, with $R_1=R$, associated with $\Gamma$. Recall that the intersection numbers $p^k_{ij}$ of the scheme depend only on the indices $k, i$, and $j$. Pick a vertex $x\in X$. 
The number of neighbors of $x$ is $|\Gamma_1(x)|=|\{z\in X \mid (x,z)\in R_1\}|$, which is equal to $p^0_{11}=n_1$. 
Now, pick two adjacent vertices of $\Gamma$, i.e., $(x,y)\in R_1$. Then, the number of common neighbors of $x$ and $y$ is $|\Gamma_1(x) \cap \Gamma_1(y)|=|\{z\in X \mid (x,z)\in R_1 \hbox{ and } (z,y)\in R_1\}|$, i.e, $p^1_{11}$. 
This completes the proof in case $\Gamma$ is a complete graph. Then, assume that $\Gamma$ is a connected strongly-regular graph for which the corresponding association scheme has 2 classes, i.e., $R_1=R$ and $R_2=\overline{R}$. Only the constant $\mu$ remains to be computed. Pick two nonadjacent vertices of $\Gamma$, i.e., $(x,y)\in R_2$. Then, the number of common neighbors of $x$ and $y$ is equal to $|\Gamma_1(x) \cap \Gamma_1(y)|=|\{z\in X \mid (x,z)\in R_1 \hbox{ and } (z,y)\in R_1\}|$, which is $p^2_{11}$.

(ii) Suppose that $\Gamma=(X,R)$ is not complete (otherwise $A=J-I$, from which the result trivially follows),  and let  $\XXi=(X, \{R_0,R,\overline{R}\})$ be its corresponding $2$-class association scheme with $\M=\langle I, A, \overline{A}\rangle$. Then, it suffices to show that $\overline{A}$ is a polynomial in $A$. By definition, $\overline{A}=J-I-A$, where $J$ can be written as a polynomial in $A$ in virtue of  \cite[Theorem~1]{Haj}. This completes the proof.
\end{proof}

\section{Some algebraic properties of $\boldsymbol{\M}$}
\label{da}

In this section we prove some results that can be found implicitly (or explicitly) in the literature. Without this section our paper is not readable as we want it to be.

Let $\M$ denote the Bose--Mesner algebra of a commutative $d$-class association scheme. In the next few claims we survey basic algebraic properties under which a matrix $A\in\M$ (not necessarily a $01$-matrix) generates $\M$. Main results that we use latter in the paper are Lemmas~\ref{dk}, \ref{ch} and Proposition~\ref{do}. Note that if $A$ is a $01$-matrix then $A$ is an adjacency matrix of some (directed) graph $\G$; we study the combinatorial structure of such a graph in Sections~\ref{pa} and \ref{6G}. Lemma~\ref{dk} can be implicitly found in some undergraduate textbook.

\begin{lemma}
\label{dk}
Let $\M$ denote the Bose--Mesner algebra of a commutative $d$-class association scheme. If $A\in\M$ has $d+1$ distinct eigenvalues, then $\{A^0,A^1,\ldots,A^d\}$ is a linearly independent set. Moreover, $\M=(\langle A\rangle, +, \cdot)$.
\end{lemma}

\begin{proof}
We show that $E_i\in\Span\{A^0,A^1,\ldots,A^d\}$ $(0\le i\le d)$. Let $\{\lambda_0,\lambda_1,\ldots,\lambda_d\}$ denote the set of the distinct eigenvalues of $A$. Using the notation from Subsection~\ref{COM}, we can permute the index set of $\lambda_i$'s and get $AU_i=\lambda_i U_i$. By the property (eii) (see page~\pageref{cl}), this yields
\begin{align*}
A&= A\sum_{i=0}^d E_i\\
&= \sum_{i=0}^d AU_i\ol{U_i}^\top\\
&=\lambda_0 E_0 + \lambda_1 E_1 + \cdots + \lambda_d E_d.
\end{align*}
We can now get the following system 
$$
A^\ell=\lambda_0^\ell E_0 + \lambda_1^\ell E_1 + \cdots +\lambda_d^\ell E_d
\qquad(0\le \ell\le d).  
$$
Note that the above system can be written as
$$
\left [
\begin{matrix} 
I\\ A\\ A^2\\ \vdots\\ A^d
\end{matrix}
\right]=
{\underbrace{
\left[
\begin{matrix} 
1 & 1 & ... & 1\\
\lambda_0 & \lambda_1 & ... & \lambda_d\\
\lambda_0^2 & \lambda_1^2 & ... & \lambda_d^2\\
\vdots & \vdots & \, & \vdots \\
\lambda_0^d & \lambda_1^d & ... & \lambda_d^d\\
\end{matrix}
\right]}_{=B^\top}}
\left[
\begin{matrix} 
E_0\\ E_1\\ E_2\\ \vdots\\ E_d
\end{matrix}
\right].
$$
where $B$ is a Vandermonde matrix (see, for example, \cite[page 185]{MCm}) which is invertible. The result follows.
\end{proof}

The Corollary~\ref{dm} is an elementary result in spectral graph theory.

\begin{corollary}
\label{dm}
Let $\M$ denote the Bose--Mesner algebra of a commutative $d$-class association scheme. For an arbitrary $A\in\M$ the following hold.  
\begin{enumerate}[label=\rm(\roman*)]
\item If $A$ is a symmetric $01$-matrix with $d+1$ distinct eigenvalues, then $A$ is an adjacency matrix of a \underline{connected} undirected graph $\G$.
\item If $A$ is a non-symmetric $01$-matrix with $d+1$ distinct eigenvalues, then $A$ is an adjacency matrix of a \underline{strongly connected} directed graph $\G$.
\end{enumerate}
\end{corollary}

\begin{proof}
From Lemma~\ref{dk}, $\{A^0,A^1,\ldots,A^d\}$ is a basis of $\M$ (recall $\M=\Span\{A_0,A_1,\ldots,A_d\}$). Since $\sum_{i=0}^d A_i=J$, the all-$1$ matrix $J$ belongs to $\M$, which yields $J\in\Span\{A^0,A^1,\ldots,A^d\}$. In other words, for any choice of vertices $y,z\in X$ there exists $\ell$ $(0\le \ell\le d)$ such that $(A^\ell)_{yz}\ne 0$ (otherwise $J\not\in\M$, a contradiction). Recall that the $(y,z)$ entry of $A^\ell$ represents the number of walks of length $\ell$ between $y$ and $z$. The result follows.
\end{proof}

\begin{lemma}
\label{dl}
Let $\M$ denote the Bose--Mesner algebra of a commutative $d$-class association scheme and let $A$ denote an arbitrary matrix from $\M$. Then the following hold.
\begin{enumerate}[label=\rm(\roman*)]
\item The sum of the row entries of $A$ is the same for every vertex.
\item The sum of the column entries of $A$ is the same for every vertex.
\item The sum of the row entries of $A$ is equal to the sum of the column entries of $A$ for every vertex.
\end{enumerate}
\end{lemma}

\begin{proof}
Consider the basis (of the adjacency matrices) $\{A_0,A_1,\ldots,A_d\}$ of $\M$, and pick some $A_i$. There exist complex scalars $\beta_h$ $(0\le h\le d)$ such that
$$
A_i=\beta_0 E_0 + \beta_1 E_1 + \cdots + \beta_d E_d
$$
(with the notation of Lemma~\ref{ct}, $\beta_0=n_i$ and $\beta_h=p_i(h)$ $(1\le h\le d)$). Since $E_0\jj=\jj$ and $E_i\jj=\0$ for all $1\le i\le d$, this yields $A_i\jj=\beta_0\jj$ (the sum of row entries of $A_i$ is equal to $\beta_0$, and that $\beta_0$ is a positive integer). Since $A$ is a linear combination of the $A_i$'s,  claim (i) follows.

Since each $A_i$ is a $01$-matrix, by Lemma~\ref{cl}(v) we have
$$
A_i^\top=\ol{A_i}^\top=\ol{\beta_0} E_0 + \ol{\beta_1} E_1 + \cdots + \ol{\beta_d} E_d,
$$
which yields $A_i^\top\jj=\ol{\beta_0}\jj$. Note that $\ol{\beta_0}=\beta_0$ is a positive integer. The sum of the column entries of $A_i$ is also equal to $\beta_0$. Since $A$ is a linear combination of $A_i$'s, claims (ii) and (iii) follow.
\end{proof}

Lemma~\ref{dl} also follows immediately from \eqref{ck}. Corollary~\ref{dn} is actually hidden in Biggs's book \cite{Bna}.

\begin{corollary}
\label{dn}
Let $\M$ denote the Bose--Mesner algebra of a commutative $d$-class association scheme. For an arbitrary $A\in\M$ the following hold.   
\begin{enumerate}[label=\rm(\roman*)]
\item If $A$ is a non-symmetric $01$-matrix, then there exists a polynomial $H(t)\in\RR[t]$ such that $J=H(A)$ if and only if $\G=\G(A)$ is a strongly connected directed graph.
\item If $A$ is a symmetric $01$-matrix, then there exists a polynomial $H(t)\in\RR[t]$ such that $J=H(A)$ if and only if $\G=\G(A)$ is a connected undirected graph.
\end{enumerate}
\end{corollary}

\begin{proof}
(i) By Lemma~\ref{dl}, for a given $01$-matrix $A\in\M$ we have $A\jj=A^\top\jj=k\jj$ for some $k$. The result now follows from \cite[Theorem~1]{HMc}.

(ii) The claim follows from Lemma~\ref{dl} and \cite[Theorem~1]{Haj}.
\end{proof}

Recall that $A$ generates $\M$ if every element of $\M$ can be written as a polynomial in $A$. One direction of the proof of Corollary~\ref{ds} is already given in Lemma~\ref{dk}.

\begin{corollary}
\label{ds}
Let $\M$ denote the Bose--Mesner algebra of a commutative $d$-class association scheme and let $A$ denote a $01$-matrix in $\M$. Then $A$ generates $\M$ if and only if $A$ has $d+1$ distinct eigenvalues.
\end{corollary}

\begin{proof}
$(\La)$ If $A$ has $d+1$ distinct eigenvalues the result follows from Lemma~\ref{dk}.

$(\Ra)$ Assume that $A$ generates $\M$. Then the matrix $J$ is polynomial in $A$. By Corollary~\ref{dn} the graph $\G=\G(A)$ is (strongly) connected (directed) graph. For the moment let $r+1$ denote the number of distinct eigenvalues of $\G$. We know that $\{A^0,A^1,\ldots,A^r\}$ is linearly independent set (see, for example, \cite[Proposition~5.4]{SPm}). Since every element of the standard basis $\{A_0,A_1,\ldots,A_d\}$ of $\M$ can be written as polynomial in $A$, we can conclude $d\le r$. On the other hand, since $A$ belong to $\M$, $A^h\in\M=\Span\{A_0,A_1,\ldots,A_d\}$ $(0\le h\le r)$  which yields $r\le d$. The result follows.
\end{proof}

The idea beyond Lemma~\ref{ch} is not so complicated: if $M$ and $N$ are simultaneously diagonalizable and $S = \alpha M+\beta N$, the eigenvalues of $S$ are known.

\begin{lemma}
\label{ch}
Let $\M$ denote the Bose--Mesner algebra of a commutative $d$-class association scheme with adjacency matrices $\{A_i\}_{i=0}^d$. For any complex scalars $\alpha_i$ $(0\le i\le d)$ the set of eigenvalues of a matrix $A=\sum_{i=0}^d \alpha_i A_i$ is
\begin{equation}
\label{db}
\{(P)_{0*}\aalpha,(P)_{1*}\aalpha,\ldots,(P)_{d*}\aalpha\}
\end{equation}
where $(P)_{i*}$ denotes the $i$th row of the first eigenmatrix $P$, and $\aalpha=(\alpha_0,\alpha_1,\ldots,\alpha_d)^\top$. (Note that we do not know the cardinality of \eqref{db}, i.e., we do not know whether the products $(P)_{i*}\aalpha$ $(0\le i\le d)$ are pairwise distinct.)
\end{lemma}

\begin{proof}
With reference to {\rm Definition~\ref{ci}}, let $\B_i$ $(0\le i\le d)$ denote a basis of the maximal common eigenspace $V_i$ $(0\le i\le d)$ of $A_0,A_1,\ldots,A_d$. For any $u\in\B_h$ $(0\le h\le d)$ we have
\begin{align*}
Au &= \sum_{i=0}^d \alpha_i A_iu\\
&= \sum_{i=0}^d \alpha_i p_i(h)u\\
&= \big((P)_{h*}\aalpha\big) u.
\end{align*}
The result follows.
\end{proof}

\smallskip
Note that the vector $\aalpha$ from Lemma~\ref{ch} is arbitrary. For our goal, the interesting case is when this vector is a $01$-vector (with $\alpha_0=0$) (see Proposition~\ref{do}).

For the moment let $A=A_1+A_2+\cdots+A_t$, where each $A_i$ $(0\le i\le t)$ is a relation (adjacency) matrix of an association scheme $\XXi$. In \cite{TI}, {\sc Ito} asked: when does the matrix $A$ with the usual matrix multiplication generate the Bose--Mesner algebra of $\XXi$, where $A$ is equal to the sum of some relation (adjacency) matrices of an association scheme $\XXi$? In Proposition~\ref{do}, we give an answer to this question. Some results on symmetric association schemes generated by a relation (case $t=1$) can be found in \cite{xTKj}. With reference to Lemma~\ref{ch}, Proposition~\ref{do} simply points out the fact that the eigenvalues of $A_i$ are given in $i$th column of the eigenmatrix $P$.

\begin{proposition}
\label{do}
Let $\XXi$ denote a commutative $d$-class association scheme with adjacency matrices $\{A_i\}_{i=0}^d$, first eigenmatrix $P$, and let $(P)_{*i}$ $(1\le i\le d)$ denote the $i$th column of $P$. The column vector $\sum_{i\in\Phi} (P)_{*i}$ has $d+1$ distinct entries (for some set of indices $\Phi\subseteq\{1,2,\ldots,d\}$) if and only if the matrix $A=\sum_{i\in\Phi} A_i$ generates the Bose--Mesner algebra $\M$ of $\XXi$.
\end{proposition}

\begin{proof}
$(\Ra)$ Assume that the column vector $\sum_{i\in\Phi} (P)_{*i}$ has $d+1$ distinct entries (for some set of indices $\Phi\subseteq\{1,2,\ldots,d\}$). Then, by Lemma~\ref{ch}, $A=\sum_{i\in\Phi} A_i$ has $d+1$ distinct eigenvalues. The result now follows from Lemma~\ref{dk}.

$(\La)$ Assume now that $A=\sum_{i\in\Phi} A_i$ generates $\M$. Then, by Corollary~\ref{ds}, $A$ has $d+1$ distinct eigenvalues. It follows from Lemma~\ref{ch} that the column vector $\sum_{i\in\Phi} (P)_{*i}$ has $d+1$ distinct entries.
\end{proof}



\section{On $\boldsymbol{3}$-class association schemes generated by a graph}
\label{ra}

In the current section, we answer the question of whether a commutative $3$-class association scheme $(X, \{R_i\}^{3}_{i=0})$ (not necessarily symmetric) is generated by a \emph{graph in the scheme}, i.e., by a graph $G_{\Phi}=(X, \{R_i\}_{i\in \Phi})$ with some nonempty index set $\Phi\subseteq \{1,\dots, d\}$.

%
%

Let $(X, \{R_i\}^{d}_{i=0})$ denote a commutative $d$-class association scheme, and consider the graph in the scheme $G_{\Phi}=(X, \{R_i\}_{i\in \Phi})$, with a nonempty set of indices $\Phi\subseteq \{1,\dots, d\}$.
With reference to Lemma~\ref{ct} and Lemma~\ref{ch}, for every $0\le \ell \le d$, the number $\sum_{i\in \Phi}p_i(\ell)$ is the eigenvalue of such a graph associated with the $\ell$th (maximal common) eigenspace of the scheme, where $p_i(\ell)$ (the $(\ell,i)$-entry of the first eigenmatrix $P$) is the eigenvalue of $(X,R_i)$ for the $\ell$th eigenspace. 
To remain consistent with our notation, we index the columns (resp. rows) of $P$ with the ordered set $(R_0,\ldots,R_d)$ (resp. $(V_0,\ldots,V_d)$) such that the $i$th column (resp. row) represents the $i$th relation (resp. eigenspace). 
For sake of simplicity, when $\Phi=\{i\}$, for some $i$ ($1 \leq i \leq d$),  we denote the graph $G_{\{i\}}=(X, R_i)$ by $G_i$.

Going back to our goal of determining if a commutative $3$-class scheme is generated by a graph, we need to distinguish two cases: 
\begin{itemize}
\item symmetric $3$-class schemes, 
\item non-symmetric $3$-class schemes.
\end{itemize}

The result in Proposition~\ref{do}, for $d=3$, will play a fundamental role in both cases.

\subsection{Symmetric $\boldsymbol{3}$-class association schemes}

Let's start with the symmetric case. 

In \cite{vDe3}, {\sc van Dam} completely classified symmetric $3$-class schemes.

\begin{lemma}[{\cite[Section 7]{vDe3}}]\label{classDam}
Let $\XXi=(X, \{R_i\}^{3}_{i=0})$ denote a symmetric $3$-class scheme, and let $G_{i}=(X,R_i)$ $(1\leq i\leq 3)$. Then, $\XXi$ belongs to exactly one of the following categories:
\begin{itemize}
\item[$(a)$] $3$-class schemes which are amorphic, i.e., every graph $G_{i}$ $(1\leq i\leq 3)$ is strongly-regular;
\item[$(b)$] $3$-class schemes with at least one graph $G_{i}$ that is the disjoint union of $N>1$  $(N \in \NN)$ connected strongly-regular graphs (which are not complete graphs) with the same parameters;
\item[$(c)$] $3$-class schemes with at least one graph $G_{i}$ having $4$ distinct eigenvalues.
\end{itemize}
\end{lemma}

In Lemma~\ref{aP} we describe the shape of the first eigenmatrix of the first class of graphs from Lemma~\ref{classDam}.

\begin{lemma}\label{aP}
Let $\XXi$ denote an amorphic symmetric $3$-class scheme. Then, its first eigenmatrix $P$ has the following form:
\[P= \,
\begin{blockarray}{cccc}
R_0 & R_1 & R_2 & R_3\\
\begin{block}{(cccc)}
1 & n_1 & n_2 & n_3 \\
1 & p_1(1) & p_2(1) & p_3(1)\\
1 & p_1(1) & p_2(2) & p_3(2)\\
1 & p_1(3) & p_2(2) & p_3(1)\\
\end{block}
\end{blockarray}\, ,
\]
where $n_i$ is the $i$th valency of the scheme.
\end{lemma}
\begin{proof}
Note that, as $P$ is not singular, the rows of $P$ are pairwise different. Furthermore, since the entries of the first column of $Q$ are all one (see Lemma~\ref{cu}), the identity $P Q=|X| I$ (Lemma \ref{cp}) implies that each row of $P$, except for the $V_0$-row, has sum zero (see Corollary~\ref{cp}). The result follows (see \cite[page 76]{vDe3} for further details).  
\end{proof}

\bigskip
In Lemma~\ref{dP} and Proposition~\ref{Gp} we we are dealing with the second class of graphs from Lemma~\ref{classDam}.

\begin{lemma}\label{dP}
Let $\XXi=(X, \{R_i\}^{3}_{i=0})$ denote a symmetric $3$-class scheme. Assume that $G_1=(X,R_1)$ is the disjoint union of $N>1$ $(N \in \NN)$ connected, non-complete strongly-regular graphs, each one with valency $k$, eigenvalues $r$ and $s$ $(r \ne k, s\ne k)$, and $w$ vertices. Then, the following hold.

\begin{enumerate}[label=\rm(\roman*)]
\item The first eigenmatrix $P$ of $\XXi$ has the form
\[P= \,
\begin{blockarray}{cccc}
R_0 & R_1 & R_2 & R_3\\
\begin{block}{(cccc)}
1& k& -1 - k + w& (-1 + N) w\\
1& k& -1 - k + w&-w\\
1& r& -1 - r& 0\\
1&s&-1 - s& 0\\
\end{block}
\end{blockarray}\, 
\]
\item None of the graphs $G_{i}=(X,R_i)$ $(1\leq i\leq 3)$ generates the scheme $\XXi$
\item The graph $G_{\{1,2\}}=(X, \{R_1, R_2\})$ does not generate the scheme $\XXi$.
\end{enumerate}
\end{lemma}
\begin{proof}
(i) The shape of $P$ is given in the first table of \cite[page 88]{vDe3}.  The proof is done by construction in \cite[Section 3.1]{vDe3}. In the next few lines, we provide a detailed description of the first eigenmatrix $P$. The  $R_1$-column of $P$ represents (the eigenvalues of) the graph which is the disjoint union of $N$ strongly-regular graphs with the same parameters and (distinct) eigenvalues $k$, $r$, $s$; the $R_2$-column represents the disjoint union of the complements of the $N$ strongly-regular graphs (with eigenvalues $-1 - k + w$, $-1-r$, $-1-s$); the $R_3$-column represents the complete $N$-partite graph (on the $N$ sets of $w$ vertices).
 
(ii) Since we are dealing with connected (non-complete) strongly-regular graphs, we may assume $r\ge0$ and $s<0$ such that $s \ne -1$ (by \eqref{rs1} and Remark \ref{r_0}). It is evident that none of the graphs $G_{i}=(X, R_i)$ $(1\leq i\leq 3)$ generates the scheme since each one has at most 2 distinct eigenvalues (for example, from the first eigenmatrix $P$, the eigenvalues of $G_{2}=(X,R_2)$ are $-1 - k + w$, $-1-r$, $-1-s$; but the first and the third ones are equal when $r=0$, as $s=k-w$ by \eqref{r_zero}). 

(iii) The graph $G_{\{1,2\}}=(X,\{R_1, R_2\})$ has $2$ distinct eigenvalues, namely $w-1$ and $-1$, obtained by adding the $R_1$-column and the $R_2$-column of $P$. Thus, $G_{\{1,2\}}=(X,\{R_1, R_2\})$ cannot generate the scheme, so proving the claim.
\end{proof}

Note that, in light of our aim (to figure out if there is a generating graph in the scheme), part (ii) allows us to set aside all graphs that we do not need to consider (see next result).

{
\begin{proposition}
\label{Gp}
Let $\XXi=(X,\{R_i\}^{3}_{i=0})$ denote a symmetric $3$-class scheme, and let $G_{i}=(X,R_i)$ $(1\leq i\leq 3)$. Assume that $\XXi$ has at least one graph $G_{i}$ (we can set $i=1$) which is the disjoint union of $N>1$  $(N \in \NN)$ connected strongly-regular graphs (which are not complete graphs) with the same parameters: valency $k$, eigenvalues $r$ and $s$ $(r \ne k, s\ne k)$, and $w$ vertices. If $r>0$ then $\XXi$ can be generated only by the following two graphs: $G_{\{1,3\}}$ and $G_{\{2,3\}}$.
\end{proposition}
}

\begin{proof}
By Lemma \ref{dP}, the first eigenmatrix $P$ has the following form:
\[P= \,
\begin{blockarray}{cccc}
R_0 & R_1 & R_2 & R_3\\
\begin{block}{(cccc )}
1& k& -1 - k + w& (-1 + N) w \\
1& k& -1 - k + w&-w \\
1& r& -1 - r& 0 \\
1&s&-1 - s& 0\\
\end{block}
\end{blockarray}\, ,
\]
where $k$, $r$, $s$ are the distinct eigenvalues of a connected, non-complete strongly-regular graph of valency $k$ on $w$ vertices, with $r\ge0$ and $s<0, \neq-1$ (see \eqref{rs1} and Remark \ref{r_0}); the parameter $N$($>1$) is the number of the (connected, non-complete) strongly-regular graphs, with the same parameters, which appear in the definition of the scheme. In light of Lemma \ref{dP}(ii), (iii), two graphs remain to be checked: $G_{\{i,3\}}=(X,\{R_i, R_3\})$, $i\in\{1,2\}$.

Recall that $r>0$. First, consider the graph $G_{\{1,3\}}=(X,\{R_1, R_3\})$, whose eigenvalues (arising from the sum of the $R_1$-column and the $R_3$-column in $P$) are $k + (-1 + N) w$, $k - w$, $r$, $s$. We verify that they are distinct from each other in the next few lines.  Note that $w>k$ (by \eqref{w}), $k>r$ (this is a well-known fact, see \cite[Section 2.3]{DP} for example), and $r>0>s$ (by \eqref{rs1}). Thus, the first eigenvalue $k + (-1 + N) w$ cannot be equal to any of the other ones, since it is positive and greater than each of the others ( $(-1 + N) w>0$). The second $k - w$, being negative ($w>k$), could be equal to $s$. Then, to get a contradiction, assume $k - w=s$. Following Subsection \ref{SRG}, for a connected strongly-regular graph, only one of the two cases can occur:
\begin{itemize}
\item[(a)] $w=2k+1$, \ $r,s=\frac{-1 \pm \sqrt{w}}{2}$ (see {\sc Case 1} at page \pageref{case1}, and \eqref{rs0});
\item[(b)] $w=\frac{(k - r) (k - s)}{k + r s}$, \ $r,s$ $\in \mathbb{Z}$ (see {\sc Case 2} at page \pageref{case2}, and  (\ref{para})).
\end{itemize}

Suppose (a) holds. Then, applying the appropriate substitutions in the equation $k - w=s$, we {get $-k-1=\frac{-1-\sqrt{w}}{2}$, and since $w=2k+1$ this yields $2k+1=\sqrt{2k+1}$. Thus we} obtain $k=0$, which is impossible by definition ($k>0$ is the valency of the graph). Now, assume (b) holds. If we replace $w$ in the equation $k - w=s$ with the expression provided in (b), then we {have $k(k+rs)-(k-r)(k-s)=s(k+rs)$ which yields $(rs+r)(k-s)=0$. Thus we} get three solutions: $k=s$ or $r=0$ or $s=-1$. All of them would obviously yield a contradiction (indeed, we have $k>0$, $r>0$, $s< - 1$ by \eqref{rs1}). 
Therefore, it follows that the graph $G_{\{1,3\}}=(X,\{R_1, R_3\})$ has $4$ distinct eigenvalues, thus generating the scheme.

Now, consider the graph $G_{\{2,3\}}=(X,\{R_2, R_3\})$ {(by assumption recall that $r>0$). Eigenvalues of $\G_{\{2,3\}}$} (arising from the sum of the $R_2$-column and the $R_3$-column in $P$) are $-1 - k + N w$, $-1 - k$, $-1 - r$, $-1 - s$. Applying the same arguments as before {(or, for example, putting equality between some of them to get a contradiction)}, it turns out that they are all distinct unless $k - N w = s$ holds. Suppose $k - N w = s$. Again, we need to distinguish the two cases (a) and (b). If (a) holds, then {from $\frac{-1-\sqrt{2k+1}}{2}=k-N(2k+1)$ we get $2N(2k+1)=2k+1+\sqrt{2k+1}$. Thus we have two solutions:} $N=\frac{1 + 2 k + \sqrt{1 + 2 k}}{2 (1 + 2 k)}$ or $k=-\frac{1}{2}$. Both of them are impossible since $N>1$ and $k>0$. If (b) holds, then {$s=k-N\frac{(k-r)((k-s))}{k+rs}$ yields $(s-k)(k+rs-N(k-r))=0$. From it} two solutions arise: $k=s$ or $N = \frac{k + r s}{k - r}$. As $s<0<r$ (by \eqref{rs1}) and $k + r s<k - r$ (note that $r$, $s$ are integers), both of them are not acceptable. Thus, also the graph $G_{\{2,3\}}=(X,\{R_2, R_3\})$ has $4$ distinct eigenvalues, thus generating the scheme. {The result follows.}
\end{proof}

\begin{theorem}
\label{CaseSym}
Let $\XXi$ denote a symmetric $3$-class scheme. Then, the scheme $\XXi$ is generated by a (undirected) graph if and only if it is not amorphic.
\end{theorem}


\smallskip
\begin{proof} We take advantage of the Van Dam classification given in Lemma \ref{classDam}.

{\sc Case 1}. Assume that $\XXi=(X, \{R_i\}^{3}_{i=0})$  is a symmetric $3$-class scheme which is amorphic, i.e., every graph $G_{i}=(X, R_i)$ $(1\leq i\leq 3)$ is a strongly-regular graph (not necessarily connected). Such a scheme is never generated by a graph in the scheme: we prove that any time we take the union of 2 classes, we get a graph with at most $3$ distinct eigenvalues. In order to understand this, it is enough to look at the first eigenmatrix $P$ of the scheme $\XXi$, whose form is provided by Lemma \ref{aP}:
\begin{equation*}P= \,
\begin{blockarray}{cccc}
R_0 & R_1 & R_2 & R_3\\
\begin{block}{(cccc)}
1 & n_1 & n_2 & n_3 \\
1 & p_1(1) & p_2(1) & p_3(1)\\
1 & p_1(1) & p_2(2) & p_3(2)\\
1 & p_1(3) & p_2(2) & p_3(1)\\
\end{block}
\end{blockarray}\, .
\end{equation*}
None of the graphs $G_{i}=(X,R_i)$ $(1\leq i\leq 3)$ generates the scheme since each one has at most $3$ distinct eigenvalues. For example, the eigenvalues of $G_{1}=(X,R_1)$ are $n_1$, $p_1(1)$, $p_1(3)$; note that $2$ of them can be equal to each other but not all $3$. Now, consider the graph $G_{\{1,2\}}=(X,\{R_1,R_2\})$, whose eigenvalues are obtained by adding the $R_1$-column and the $R_2$-column of $P$; precisely, they are $n_1+n_2$,  $p_1(1)+p_2(1)$, $p_1(1)+p_2(2)$, $p_1(3)+p_2(2)$. Since the sum of the $R_0$-row of $P$, $1+n_1+n_2+n_3$, equals $n$ (the number of vertices of the scheme) and the sum of any other row of $P$ is zero {(see Corollary~\ref{cq})}, these eigenvalues are respectively equal to $n-n_3-1$, $-1-p_3(1)$, $-1-p_3(2)$, $-1-p_3(1)$. This yields that the graph $G_{\{1,2\}}=(X,\{R_1,R_2\})$ has at most $3$ distinct eigenvalues, namely, $n-n_3-1$, $-1-p_3(1)$, $-1-p_3(2)$. Thus, the graph $G_{\{1,2\}}=(X,\{R_1,R_2\})$ cannot generate the scheme. The same conclusions arise if we choose the remaining graphs $G_{\{1,3\}}=(X,\{R_1,R_3\})$ and $G_{\{2,3\}}=(X,\{R_2,R_3\})$.

{\sc Case 2}. Suppose that $\XXi$ has a graph $G_{i}=(X,R_i)$ which is the disjoint union of connected strongly-regular graphs with the same parameters {which are not complete graphs}. We may assume $i=1$. For $r>0$, by Proposition~\ref{Gp} each of $G_{\{1,3\}}$ and $G_{\{2,3\}}$ generate the scheme. For $r=0$, {in the next few lines,} we prove that only $G_{\{2,3\}}=(X,\{R_2, R_3\})$ has $4$ distinct eigenvalues, thus generating the scheme.

Assume $r=0$. Then, by \eqref{r_zero}, the first eigenmatrix $P$ is as follows:
\[P= \,
\begin{blockarray}{cccc}
R_0 & R_1 & R_2 & R_3\\
\begin{block}{(cccc )}
1& k& -1 + k - \lambda& (-1 + N) (2 k - \lambda)\\
1& k& -1 + k - \lambda&- (2 k - \lambda) \\
1& 0& -1 & 0 \\
1&-k+\lambda& -1 + k -\lambda& 0\\
\end{block}
\end{blockarray}\, .
\]
In this case, the graph $G_{\{1,3\}}=(X,\{R_1, R_3\})$ has 3 (distinct by \eqref{r_zero}) eigenvalues, i.e., $k + (-1 + N) (2k -\lambda)$, $-k + \lambda$, $0$; so it cannot generate the scheme.  Then, consider the graph $G_{\{2,3\}}=(X,\{R_2, R_3\})$, whose eigevalues are $-1 - k + N (2 k - \lambda)$, $-1 - k$, $-1$, $-1 + k - \lambda$. The last three are all different from each other as $k>0$, $w=2k-\lambda>0$, and $-s=k-\lambda>0$ (see \eqref{r_zero}). Now, comparing the first eigenvalue with each of the others, we never get an equality since $N (2 k - \lambda)=N w >0$, $k<N w$, and $(-1+N) w >0$, respectively (by $N>1$, $N \in \NN$, and \eqref{r_zero}). This means that the graph $G_{\{2,3\}}=(X,\{R_2, R_3\})$ has 4 distinct eigenvalues, and so it generates the scheme.

{\sc Case 3}. Suppose that $\XXi$ has a graph $G_{i}=(X,R_i)$ having $4$ distinct eigenvalues. Then, $G_{i}$ generates $\XXi$.
\end{proof}



\subsection{Non-symmetric $\boldsymbol{3}$-class association schemes}

Now, we cnsider the non-symmetric case.

In \cite{GCs}, {\sc Goldbach} found the general structure of the first eigenmatrix $P$ of a non-symmetric $3$-class scheme.

\begin{lemma}[{\cite[Theorem 2.3]{GCs}}]\label{gold} Let $\XXi$ denote a non-symmetric $3$-class scheme with $n$ points, intersection numbers $p^k_{ij}$ $(0\leq k,i,j\leq 3)$, valencies $n_i=p^0_{ii}$ $(0\leq i \leq 3)$, and multiplicities $m_i$ $(0\leq i \leq 3)$. Then, the first eigenmatrix $P$ of the scheme $\XXi$ has the following form:

\begin{equation}\label{PGold}
P= \,
\begin{blockarray}{cccc}
R_0 & R_1 & R_2=R_1^{\top} & R_3\\
\begin{block}{(cccc)}
1 & n_1 & n_1 & n_3\\
1 & p_1(1) & \overline{p_1(1)} & p_3(1)\\
1 & \overline{p_1(1)} & p_1(1) & p_3(1)\\
1 & p_1(3)& p_1(3) & p_3(3)\\
\end{block}
\end{blockarray}\, ,
\end{equation}

where $p_1(1)=\frac{1}{2}(p^1_{11}-p^2_{11}+{\rm i}\sqrt{\frac{n n_1}{m_1}}) \in {\mathbb C \setminus \mathbb{R}}$ and $p_1(3)=\frac{n_1}{n_3}p^1_{33}-p^1_{23} \in \mathbb{Q}$. 

\end{lemma}

\begin{remark}\label{Prmk}
\emph{With reference to Lemma \ref{gold}, let us make some considerations on the first eigenmatrix $P$. Since the sum of every row of $P$ except the $V_0$-row is zero {(see Corollary~\ref{cq})}, it turns out that $p_3(1)=p^2_{11}-p^1_{11}-1 \in \mathbb{Z}$ and $p_3(3)=-1-2p_1(3)\in \mathbb{Q}$. Furthermore, as $p^1_{11}-p^2_{11}+p^1_{13}-p^1_{23}=-1$ (see \cite[Lemma 2.4]{GCs}), we can write $p_3(1)=p^1_{13}-p^1_{23}$. In the end, note that $n_1\neq p_1(1)$ and $p_1(3)\neq p_1(1)$ since $n_1\in \NN$, $p_1(1)\in \CC \setminus \RR$, and $p_1(3)\in \QQ$.}
\end{remark}

To make the next arguments clearer, let us recall the following definition from \cite[page~79]{BI}. 
An association scheme $(X, \{R_i\}^{d}_{i=0})$ is said to be \emph{primitive} if every graph $G_{i}=(X, R_i)$ $(1\leq i \leq d)$ is connected; otherwise, it is said to be \emph{imprimitive}.

\begin{theorem}
\label{CaseAsym}
Let $\XXi$ be a non-symmetric $3$-class scheme. Then, there always exists a directed graph which generates the scheme $\XXi$.
\end{theorem}

\begin{proof} Since $\XXi=(X, \{R_i\}^{3}_{i=0})$ is a non-symmetric $3$-class scheme, its first eigenmatrix $P$ looks like the one in (\ref{PGold}). Thus, we will use here the same notation as in Lemma \ref{gold} as well as the contents of Remark \ref{Prmk}. Two cases arise: the scheme $\XXi$ is primitive or imprimitive.

We first suppose $\XXi$ is primitive, i.e., each of its relations but the diagonal one is connected. 
Recall that the association scheme $(X, \ol{\R})$, in which $\ol{\R}=\{R\cup R^{\top} \mid R\in \R\}$, is said to be the \emph{symmetric closure} of the scheme $(X, \R)$. In our case, $(X, \{R_1\cup R_2, R_3\})$ is the symmetric closure of $\XXi$. By \cite[Theorem~2.2]{GCs}, a $3$-class scheme is primitive if and only if its symmetric closure is primitive. This implies that the (undirected) graph $(X, R_1\cup R_2)$ is connected with eigenvalues $2 n_1$, $p_1(1)+\ol{p_1(1)}$, and $2p_1(3)$ (see \eqref{PGold}). It is known that the number of connected components of a regular (undirected) graph is the multiplicity of its valency (see \cite[page~1]{BVM}). Since the valency of the graph $(X, R_1\cup R_2)$ is $2n_1$, we have that $2n_1\ne 2p_1(3)$, i.e., $n_1\neq p_1(3)$. It follows that the entries in the $R_1$-column of $P$ are distinct, that is, the graph $G_{1}=(X, R_1)$ has $4$ distinct eigenvalues, and so it generates the scheme. The same arguments hold for $G_{2}=(X,R_2)$. Observe that $G_{3}=(X,R_3)$ has at most $3$ distinct eigenvalues, and so it cannot generate the scheme.

We explore now the case in which $\XXi$ is imprimitive. According to \cite[Theorem~4.1]{GCs}, this means that $p^1_{33}(p^1_{13}+p^1_{23})=0$. Since the intersection numbers $p^k_{ij}$ are non-negative integers by definition, then either $p^1_{33}=0$ or $p^1_{13}=p^1_{23}=0$, which never occur together, otherwise $n_3=p^1_{31}+p^1_{32}+p^1_{33}$ (see equation \eqref{ck}) would be zero.

If $p^1_{33}=0$, the first eigenmatrix $P$ appears as follows:
\[
P= \,
\begin{blockarray}{cccc}
R_0 & R_1 & R_2=R_1^{\top} & R_3\\
\begin{block}{(cccc)}
1 & n_1 & n_1 & n_3\\
1 & p_1(1) & \overline{p_1(1)} & p_3(1)\\
1 & \overline{p_1(1)} & p_1(1) & p_3(1)\\
1 & -p^1_{23}& -p^1_{23} & -1+ 2 p^1_{23}\\
 \end{block}
\end{blockarray}\, ,
\]
where $n_1\neq -p_{23}^1$. Then, $G_{i}=(X,R_i)$, $i\in\{1,2\}$, having $4$ distinct eigenvalues, generates the scheme. Note that $G_{3}=(X, R_3)$ is disconnected.

If $p^1_{13}=p^1_{23}=0$, then $n_3=p^1_{33}$ and $p_3(1)=0$. The first eigenmatrix $P$ is now the following:
\[
P= \,
\begin{blockarray}{cccc}
R_0 & R_1 & R_2=R_1^{\top} & R_3\\
\begin{block}{(cccc)}
1 & n_1 & n_1 & n_3\\
1 & p_1(1) & \overline{p_1(1)} & 0\\
1 & \overline{p_1(1)} & p_1(1) & 0\\
1 & n_1& n_1 & -1-2 n_1\\
 \end{block}
\end{blockarray}\, .
\]
None among the graphs $G_{i}=(X,R_i)$ $(1\leq i\leq 3)$ can generate the scheme, as each of them has exactly $3$ distinct eigenvalues. Let us then consider the graph $G_{\{1,3\}}=(X, \{R_1,R_3\})$, whose eigenvalues are obtained by adding the $R_1$-column and the $R_3$-column of $P$. Thus, this graph has $4$ distinct eigenvalues, namely $n_1+n_3$, $p_1(1)$, $\overline{p_1(1)}$, $-1-n_1$, thus generating the scheme. The same holds if we consider $G_{\{2,3\}}=(X, \{R_2,R_3\})$. Note that $G_{\{1,2\}}=(X, \{R_1,R_2\})$ has distinct eigenvalues $2n_1$ and $-1$; hence it cannot generate the scheme.
\end{proof}

\subsection{Proof of Theorem~\ref{ph}} 

Let $\XXi$ be a commutative $3$-class association scheme. If $\XXi$ is symmetric, then the result follows from Theorem \ref{CaseSym}. Otherwise, $\XXi$ is non-symmetric and the result follows from Theorem \ref{CaseAsym}.


\section{The distance-faithful intersection diagram}
\label{pa}

Let $\M$ denote the Bose--Mesner algebra of a commutative $d$-class association scheme $\XXi$ and $A\in\M$ denote a $01$-matrix which generates $\M$. In this section, we study combinatorial properties of $\G=\G(A)$. We prove that, whenever a $01$-matrix $A\in\M$ represents a (strongly) connected (directed) graph, then for every vertex $x\in X$ there exists an $x$-distance-faithful intersection diagram of an equitable partition $\Pi_x$ with $d+1$ cells. Moreover, the structure of the $x$-distance-faithful intersection diagram does not depend on $x$ (see Theorem~\ref{pi}). We use this fact to describe combinatorial properties of a graph which generates a commutative $3$-class association scheme (see Corollary~\ref{we}).

\begin{lemma}
\label{pg}
Let $\M$ denote the Bose--Mesner algebra of a commutative $d$-class association scheme $\XXi=(X,\R)$ with adjacency matrices $\{A_i\}_{i=0}^d$.
For a given $x\in X$ we define the partition $\Pi_x=\{\P_0(x),\P_1(x),\ldots,\P_d(x)\}$ of $X$ in the following way 
$$
\P_i(x)=\{z\mid  (A_i)_{xz}=1\}\qquad (0\le i\le d).
$$
Let $A$ denote an arbitrary $01$-matrix in $\M$, and consider the (directed) graph $\G=\G(A)$. If $\G$ is a (strongly) connected (directed) graph then in $\G$ all vertices in $\P_i(x)$ are at the same distance from $x$.
\end{lemma}

\begin{proof}
We first show that for any $z,w\in\P_i(x)$ the number of walks of length $\ell$ from $x$ to $z$ is the same as the number of walks of length $\ell$ from $x$ to $w$ (i.e., $(A^\ell)_{xz}=(A^\ell)_{xw}$ $(0\le \ell\le d)$). Since $\{A_h\}_{h=0}^d$ is a basis of $\M$, there exist scalars $\alpha_{ij}$ $(0\le i,j\le d)$ such that
$$
A^\ell=\sum_{j=0}^d \alpha_{\ell j} A_j
\qquad
(0\le \ell\le d).
$$
For any $z,w\in\P_i(x)$, we have $(A_i)_{xz}=(A_i)_{xw}=1$ and $(A_j)_{xz}=(A_j)_{xw}=0$ if $j\ne i$. This yields $(A^\ell)_{xz}=\alpha_{\ell i}=(A^\ell)_{xw}$.

We now prove our claim by a contradiction. Assume that $z,w\in\P_i(x)$ and that $\partial(x,z)>\partial(x,w)=\ell$. Then, we  have $(A^{\ell})_{xw}\ne 0$ but $(A^{\ell})_{xz}= 0$, a contradiction.
\end{proof}

\begin{lemma}
\label{pl}
Let $\M$ denote the Bose--Mesner algebra of a commutative $d$-class association scheme $\XXi=(X,\R)$ with the adjacency matrices $\{A_i\}_{i=0}^d$. Pick $x,y\in X$ and define the partitions $\Pi_x=\{\P_0(x),\P_1(x),\ldots,\P_d(x)\}$ and $\Pi_y=\{\P_0(y),\P_1(y),\ldots,\P_d(y)\}$ of $X$ in the following way: 
$$
\P_i(x)=\{z\mid  (A_i)_{xz}=1\},\quad
\P_i(y)=\{z\mid  (A_i)_{yz}=1\}
\qquad (0\le i\le d).
$$
Let $A$ denote an arbitrary $01$-matrix in $\M$, and consider the (directed) graph $\G=\G(A)$. If $\G$ is a (strongly) connected (directed) graph then for any $i,j$ $(0\le i,j\le d)$ there exist scalars $D_{ij}^{\ra}$ such that in $\G$ the following hold:
$$
|\G_1^{\ra}(z)\cap \P_j(x)|=D^{\ra}_{ij} \qquad\mbox{for every $z\in\P_i(x)$}
$$
and
$$
|\G_1^{\ra}(w)\cap\P_j(y)|= D^{\ra}_{ij} \qquad\mbox{for every $w\in\P_i(y)$}.
$$
\end{lemma}

\begin{proof}
We give a proof for a directed graph. The proof for an undirected graph is similar.

Pick some $i,j$ $(0\le i,j\le d)$ and let $k$ and $\ell$ denote the unique indices such that $A_k=A_j^\top$ and $A_i^\top=A_\ell$ (such indices exists since $A_i^\top,A_j^\top\in\{A_0,A_1,\ldots,A_d\}$). Note that $A_k^\top=A_j$ and $A_\ell^\top=A_{i}$. Since $AA_k\in\Span\{A_0,A_1,\ldots,A_d\}$, there exist scalars $\alpha^h_{k}$ $(0\le h\le d)$ such that
\begin{equation}
\label{7b}
AA_k=\sum_{h=0}^d \alpha^h_{k} A_h.
\end{equation}
Pick $x,y\in X$ and consider the partitions $\Pi_x$ and $\Pi_y$. We show that for any $z\in\P_i(x)$ and $w\in\P_i(y)$, we have $|\G_1^{\ra}(z)\cap \P_j(x)|=\alpha^\ell_{k}$ and $|\G_1^{\ra}(w)\cap \P_j(y)|=\alpha^\ell_{k}$.

Note that for any matrix $B$, $(B)_{zx}=(B^\top)_{xz}$. From the left-hand side of \eqref{7b}, we have
\begin{align*}
(AA_k)_{zx}&=\sum_{u\in X} (A)_{zu} (A_k)_{ux}\\
&=\sum_{u\in X} (A)_{zu} (A_j)_{xu}\\
&= |\G_1^{\ra}(z)\cap \P_j(x)|
\end{align*}
and
\begin{align*}
(AA_k)_{wy}&=\sum_{u\in X} (A)_{wu} (A_k)_{uy}\\
&=\sum_{u\in X} (A)_{wu} (A_j)_{yu}\\
&= |\G_1^{\ra}(w)\cap \P_j(y)|.
\end{align*}
For the same choices of $z\in\P_i(x)$ and $w\in\P_i(y)$ as above, from the right-hand side of \eqref{7b}, we have
\begin{align*}
(AA_k)_{zx} &= ({AA_k})^\top_{xz}\\
&= \left({\sum_{h=0}^d \alpha^h_{k} A_h}\right)^\top_{xz}\\
&= \sum_{h=0}^d \alpha^h_{k} (A_h^\top)_{xz}\\
&= \alpha^\ell_{k} (A_i)_{xz}\qquad\mbox{(where $(A_\ell)^\top=A_{i}$)}\\
&= \alpha^\ell_{k}
\end{align*}
and
\begin{align*}
(AA_k)_{wy} &= ({AA_k})^\top_{yw}\\
&= \left({\sum_{h=0}^d \alpha^h_{k} A_h}\right)^\top_{yw}\\
&= \sum_{h=0}^d \alpha^h_{k} (A_h^\top)_{yw}\\
&= \alpha^\ell_{k} (A_i)_{yw}\qquad\mbox{(where $(A_\ell)^\top=A_{i}$)}\\
&= \alpha^\ell_{k}.
\end{align*}
With it, if we define $D_{ij}^{\ra}$ as $\alpha^\ell_{k}$ (the index $i$ uniquely determines $\ell$, and the index $j$ uniquely determines $k$), we get that
$$
|\G_1^{\ra}(z)\cap \P_j(x)|=D^{\ra}_{ij} \qquad\mbox{for every $z\in\P_i(x)$}
$$
and
$$
|\G_1^{\ra}(w)\cap \P_j(y)|= D^{\ra}_{ij} \qquad\mbox{for every $w\in\P_i(y)$.}
$$
\end{proof}

\subsection{Proof of Theorem~\ref{pi}}


In this subsection we prove Theorem~\ref{pi}. The proof is in the same spirit as \cite[Theorem~4.1]{FMPS}. 

Assume that $A$ is a non-symmetric matrix. Using the same notation as in Lemma~\ref{pg}, for a given $x\in X$ we define the partition $\Pi_x$ on the following way: 
$$
\Pi_x=\{\P_0(x),\P_1(x),\ldots,\P_d(x)\},
\qquad\mbox{ where }\quad
\P_i(x)=\{z\mid  (A_i)_{xz}=1\}~(0\le i\le d).
$$
To prove the claim, we need to show that the following (a)--(c) hold.
\begin{enumerate}[label=\rm(\alph*)]
\item All vertices in $\P_i(x)$ are at the same distance from $x$.
\item $|\P_i(x)|=|\P_i(u)|$ $(0\le i\le d)$ for every $x,u\in X$.
\item There exist numbers $D^{\ra}_{ij}$, $D^{\la}_{ij}$ $(0\le i,j\le d)$ such that, for every $x\in X$, $\Pi_x$ is an equitable partition of $\G$ with corresponding parameters $D^{\ra}_{ij}$, $D^{\la}_{ij}$ (which do not depend on $x$).
\end{enumerate}

\medskip
The claim (a) follows immediately from Lemma~\ref{pg}.

\medskip
For the claim (b) first note that every matrix in $\M$ has constant row sums (see Lemma~\ref{dl}). Thus $|\P_i(x)|=\sum_{z\in X} (A_i)_{xz}=\sum_{w\in X} (A_i)_{uw}=|\P_i(u)|$ holds for every $x,u\in X$. (Furthermore, note that the cardinality of $\P_i(x)$ for every $x\in X$ is equal to $|R_i(x)|=|\{z\in X\mid (x,z)\in R_i \}|$ and that $|R_i(x)|=n_i$ where $R_i$ is $i$th relation of the association scheme $\XXi$ and $n_i$ is valency of $R_i$ (see Subsection~\ref{COM})).

\medskip
It is left to prove claim (c). In Lemma~\ref{pl} we showed that for any $i,j$ $(0\le i,j\le d)$ and $x,y\in X$ there exists scalars $D_{ij}^{\ra}$ such that in $\G$, $|\G_1^{\ra}(z)\cap \P_j(x)|=D^{\ra}_{ij}$ holds for every $z\in\P_i(x)$; and that $|\G_1^{\ra}(w)\cap\P_j(y)|= D^{\ra}_{ij}$ holds for every $w\in\P_i(y)$.

For $D^{\la}_{ij}$ we have something similar. Pick $i,j$ $(0\le i,j\le d)$ and $x,y\in X$. First, note that
\begin{equation}
\label{7c}
A_jA=\sum_{h=0}^d \beta^{h}_{j} A_h.
\end{equation}
for some scalars $\beta^h_{j}$ $(0\le h\le d)$. For any $z\in\P_i(x)$ and $w\in\P_i(y)$, from the left-hand side of \eqref{7c}, we have
\begin{align*}
(A_jA)_{xz}&=\sum_{u\in X} (A_j)_{xu} (A)_{uz}\\
&=\sum_{u\in\P_j(x)} (A)_{uz}\\
&=\sum_{u\in\P_j(x)} |\G^{\ra}_1(u)\cap \{z\}|
\end{align*}
and
\begin{align*}
(A_jA)_{yw}&=\sum_{u\in X} (A_j)_{yu} (A)_{uw}\\
&=\sum_{u\in\P_j(y)} (A)_{uw}\\
&=\sum_{u\in\P_j(y)} |\G^{\ra}_1(u)\cap \{w\}|.
\end{align*}
For the same choices of $z\in\P_i(x)$ and $w\in\P_i(y)$, from the right-hand side of \eqref{7c}, we have
\begin{align*}
(A_jA)_{xz}&=\left(\sum_{h=0}^d \beta^{h}_{j} A_h\right)_{xz}\\
&=\beta^{i}_{j}(A_i)_{xz}\\
&=\beta^{i}_{j}
\end{align*}
and
\begin{align*}
(A_jA)_{yw}&=\left(\sum_{h=0}^d \beta^{h}_{j} A_h\right)_{yw}\\
&=\beta^{i}_{j}(A_i)_{yw}\\
&=\beta^{i}_{j}.
\end{align*}
With it, if we define $D_{ij}^{\la}$ as $\beta^{i}_{j}$, we get that
$$
\sum_{u\in\P_j(x)} |\G^{\ra}_1(u)\cap \{z\}|=D^{\la}_{ij} \qquad\mbox{for every $z\in\P_i(x)$,}
$$
and
$$
\sum_{u\in\P_j(y)} |\G^{\ra}_1(u)\cap \{w\}|= D^{\la}_{ij} \qquad\mbox{for every $w\in\P_i(y)$.}
$$
Thus, $\Pi_x$ and $\Pi_y$ are equitable partitions of $\G$ with the same corresponding parameters $D^{\ra}_{ij},D^{\la}_{ij}$ $(0\le i,j\le d)$.

\medskip
For the end of this section, let $\FF_q$ denote the finite field of order $q$, and let $p$ denote a prime number with $q=p^m$. In \cite{QHL} the authors study under which condition the image of $p$-ary function (a function from $\FF_q$ to $\FF_p$) give an association scheme, by studying a specific partition $\FF_q$. It would be interesting to find in which way this partition is connected with our partition from Theorem~\ref{pi}, if any connection exists at all.


\subsection{Some corollaries of Theorem~\ref{pi}}

Theorem~\ref{pi} gives us a useful combinatorial property for a (strongly) connected (directed) graph which `lives' in a $d$-class association scheme. See Corollary~\ref{we} to understand what is happening in a $3$-class association scheme.

Recall that a graph is {\em walk-regular} if the number of closed walks of length $\ell$ rooted at vertex $x$ only depends on $\ell$, for each $\ell\ge 0$ (i.e., the $(A^\ell)_{xx}$ entry for every $x\in X$ only depends on $\ell$).

\begin{corollary}
\label{wf}
Let $\M$ denote the Bose--Mesner algebra of a commutative $d$-class association scheme $\XXi=(X,\R)$. If a (strongly) connected (directed) graph $\G$ `lives' in the association scheme $\XXi$ (i.e., if the adjacency matrix $A$ of $\G$ belongs to $\M$), then $\G$ is a walk-regular graph.
\end{corollary}

\begin{proof}
Immediate from Theorem~\ref{pi}.
\end{proof}

\medskip
In Corollary~\ref{wg} we deal with a symmetric $d$-class association scheme.

\begin{corollary}
\label{wg}
Let $\M$ denote the Bose--Mesner algebra of a symmetric $d$-class association scheme $\XXi=(X,\R)$, and $A\in\M$ denote a $01$-matrix. If $\G=\G(A)$ generates $\XXi$ then the following hold.
\begin{enumerate}[label=\rm(\roman*)]
\item For every vertex $x\in X$, there exists an $x$-distance-faithful intersection diagram (of an equitable partition $\Pi_x$) with $d+1$ cells.
\item The structure of the $x$-distance-faithful intersection diagram (of the equitable partition $\Pi_x$) from {\rm (i)} does not depend on $x$.
\item Graph $\G$ does not have an $x$-distance-faithful intersection diagram whose number of cells is less than $d+1$ (i.e., $d+1$ is the smallest number of cells for which there exists an $x$-distance-faithful equitable partition).
\end{enumerate}
\end{corollary}

\begin{proof}
By assumption $A$ generates $\M$, so by Corollary~\ref{ds} $A$ has $d+1$ distinct eigenvalues $\lambda_0>\lambda_1>\cdots>\lambda_d$. Note that Corollary~\ref{wf} yields that $\G$ is a walk-regular graph. By Theorem~\ref{pi}, for every vertex $x\in X$, there exists an $x$-distance-faithful intersection diagram (of an equitable partition $\Pi_x$) with $d+1$ cells and the structure of the intersection diagram does not depend on $x$ (so claims (i) and (ii) hold). For the moment let $B$ denote the $(d+1)\times(d+1)$ quotient matrix of the $x$-distance-faithful intersection diagram. By \cite[Proposition~4.1]{DF} every $\lambda_i$ $(0\le i\le d)$ is an eigenvalue of $B$. Now our proof is by a contradiction. Assume that there exists an $x$-distance-faithful intersection diagram with less than $d+1$ cells. Then quotient matrix $C$ of such intersection diagram has less than $d+1$ distinct eigenvalues, and by \cite[Proposition~4.1]{DF} every of $d+1$ distinct eigenvalues $\lambda_i$ $(0\le i\le d)$ of $\G$ are also eigenvalues of $C$, a contradiction. The claim (iii) follows.
\end{proof}

\begin{corollary}
\label{we}
Let $\M$ denote the Bose--Mesner algebra of a commutative $3$-class association scheme $\XXi=(X,\R)$, $A\in\M$ denote a $01$-matrix, and let $\G=\G(A)$ denote a (directed) graph of diameter $D$ with adjacency matrix $A$. If $\G$ generates $\XXi$ then $D\in\{2,3\}$, $\G$ has the same $x$-distance-faithful intersection diagram around every vertex $x\in X$ and such a diagram has $4$ cells. Moreover, the following hold.
\begin{enumerate}[label=\rm(\roman*)]
\item If $D=3$, then the partition $\{\G_i(x)\}_{0\le i\le 3}$ is equitable, and the corresponding parameters do not depend on the choice of $x\in X$.
\item If $D=2$, then exactly one of the following {\rm(a), (b)} holds.
\begin{enumerate}[label=\rm(\alph*)]
\item Any two adjacent vertices have a constant number of common neighbors, and the number of common neighbors of any two nonadjacent vertices takes precisely two values. Moreover, for any $x\in X$ there exists an equitable partition\break $\Pi_x=\{\{x\},\G_1(x),\P(x),\P'(x)\}$, for which $\G_2(x)=\P(x)\cup\P'(x)$. 
\item Any two nonadjacent vertices have a constant number of common neighbors, and the number of common neighbors of any two adjacent vertices takes precisely two values. Moreover, for any $x\in X$ there exists an equitable partition\break $\Pi_x=\{\{x\},\P(x),\P'(x),\G_2(x)\}\}$, for which $\G_1(x)=\P(x)\cup\P'(x)$.
\end{enumerate}
\end{enumerate}
\end{corollary}

\begin{proof}
Corollary~\ref{ds} yields that $\G$ has $4$ distinct eigenvalues, and by Corollary~\ref{dm}, $\G$ is a (strongly) connected (directed) graph.

We first show that $D\le 3$. Since $\{A^0,A^1,\ldots,A^D\}$ is a linearly independent set (this is a well-known fact, see for example \cite[Proposition~5.6]{SPm}) and since $\{A^0,A^1,\ldots,A^d\}$ is a basis of $\A$, we have $D\le d$, and consequently $D\le 3$. Next we show that $D=1$ is not possible.
If $D=1$ then every two different vertices are adjacent, which yields that $\G$ is a complete graph. Then, we have that $A=J-I$ is the adjacency matrix, which yields that $A$ has less then $4$ distinct eigenvalues, a contradiction. Case $D=1$ is not possible.

By Theorem~\ref{pi}, the number of cells of a distance-faithful equitable partition is equal to $4$. 

Assume that $D=3$. Pick $x\in X$. The only possibility to get a $x$-distance-faithful equitable partition with $4$ cells is to take distance partition $\{\G_i(x)\}_{0\le i\le 3}$ of $X$. An example of a directed graph with $D=3$ which generates $3$-class association scheme is given in Figure~\ref{2f}.

Assume that $D=2$. For the moment let $\{B_0=I,B_1,B_2,B_3\}$ denote the standard basis of $\M$, and let $A$ denote the adjacency matrix of a graph $\G=\G(A)$. Since $A\in\M$, the matrix $A$ is equal to some linear combination of $\{B_0,B_1,B_2,B_3\}$. Moreover, since $A$ and the $B_i$'s are $01$-matrices, in total six cases are possible $A\in\{B_1,B_2,B_3\}$ or $A\in\{B_1+B_2,B_1+B_3,B_2+B_3\}$. (Case $A=B_1+B_2+B_3$ is not possible since then we would have a complete graph.) First three cases $A\in\{B_1,B_2,B_3\}$ give claim (a). Cases $A\in\{B_1+B_2,B_1+B_3,B_2+B_3\}$ yield claim (b). Note that, if we do not have two different values (in both cases), then $\G$ is a strongly-regular graph, a contradiction (by assumption, $\G$ generates $\XXi$). 

The result follows.
\end{proof}

\section{Algebraic property of $\boldsymbol{\G}$ when $\boldsymbol{\G}$ generates a commutative association scheme}
\label{6G}

In this section we prove Theorem~\ref{PI}. For that purpose we need Proposition~\ref{Ha}.

\begin{proposition}
\label{Ha}
Let $\G=\G(A)$ denote a directed graph with vertex set $X$ and adjacency matrix $A$. Assume that $A$ generates the Bose--Mesner algebra $\M$ of a commutative $d$-class association scheme, and let $\{B_0,B_1,\ldots,B_d\}$ denote the standard basis of $\M$. Then, the following hold.
\begin{enumerate}[label=\rm(\roman*)]
\item For any $i$ $(0\le i\le d)$ and $y,z,u,v\in X$, if $(B_i)_{zy}=(B_i)_{uv}=1$ then $\partial(z,y)=\partial(u,v)$.
\item Every distance-$i$ matrix $A_i$ of $\G=\G(A)$ belongs to $\M$, i.e., $A_i\in\M$ $(0\le i\le D)$.
\end{enumerate}
\end{proposition}

\begin{proof}
Since $A$ generate the Bose--Mesner algebra $\M$, and $J\in\M$, there exists a polynomial $p(t)$ such that $J=p(A)$. This implies that $\G$ is regular and strongly connected (see Corollary~\ref{dn}).

(i) For every $\ell\in\NN$, there exists complex scalars $\alpha^{(\ell)}_i$ $(0\le i\le d)$ such that $A^\ell=\sum_{i=0}^{d} \alpha^{(\ell)}_i B_i$. Recall that $\sum_{i=0}^d B_i = J$ and $B_i\circ B_j=\delta_{ij}B_i$ $(0\le i,j\le d)$. This yields that for any $y,z,u,v\in X$ and $i$ $(0\le i\le d)$, if $(B_i)_{zy}\ne 0$ and $(B_i)_{uv}\ne 0$ then $(A^\ell)_{zy}=(A^\ell)_{uv}=\alpha^{(\ell)}_i$, i.e., the number of walks of length $\ell$ from $z$ to $y$ is equal to the the number of walks of length $\ell$ from $u$ to $v$ (see Lemma~\ref{hB}). 
Moreover, $(A^\ell)_{zy}=(A^\ell)_{uv}$ holds for any $\ell$ $(\ell\in\NN)$. To prove the claim, we use the proof by a contradiction, similar as in \cite[Lemma~2.3]{FQpG} where the author has an undirected graph. Assume that $\partial(z,y)>\partial(u,v)=m$. Then, $(A^m)_{uv}\ne 0$ and $(A^m)_{zy}=0$, a contradiction. 
The result follows. 

(ii) From the proof of (i) above it follows that, if $y,z\in X$ are two arbitrary vertices such that $\partial(z,y)=i$, then there exists $B_j$ (for some $0\le j\le d$) such that $(B_j)_{zy}=1$. Recall also that $(A_i)_{zy}=1$. In fact, for such a choice of $j$ and any nonzero $(u,v)$-entry of $B_j$, we have $\partial(u,v)=i$. This yields
$$
A_i=\sum_{j:A_i\circ B_j\ne{\boldsymbol{O}}} B_j
\qquad(0\le i\le D).
$$
The result follows.
\end{proof}

\subsection{Proof of Theorem~\ref{PI}}

We show that (i)$\Lra$(ii), (iii)$\Ra$(i) and (ii)$\Ra$(iii). Recall that
\begin{equation}
\label{Pj}
\Delta=\{(i,j) \mid i=\partial(x,y),\, j=\partial(y,x),\, x,y\in X \}.
\end{equation} 

\bigskip
(i)$\Ra$(ii). Assume that $\A$ is the Bose--Mesner algebra of a $d$-class association scheme $\XXi=(X,\R)$, and let $\{B_i\}_{i=0}^d$ denote adjacency matrices of $\XXi$. Note that $\A=\Span\{A^0,A^1,\ldots,A^d\}=\Span\{B_0,B_1,\ldots,B_d\}$. Assume that $\G$ has diameter $D$, and let $A_i$'s denote the distances-$i$ matrices of $\G$. For a given $x\in X$ we define a partition $\Pi_x=\{\P_0(x),\P_1(x),\ldots,\P_d(x)\}$ of $X$ in the following way 
$$
\P_i(x)=\{z\mid  (B_i)_{xz}=1\}\qquad (0\le i\le d).
$$
From the proof of Theorem~\ref{pi}, the partition $\Pi_x$ is equitable and the corresponding parameters do not depend of the choice of $x$.

Pick $i$ $(0\le i\le d)$ and $z\in\P_i(x)$; hence $(B_i)_{xz}=1$. For the moment, assume that $\partial(x,z)=h$. Note that $A_h\circ B_i=B_i$ (see Proposition~\ref{Ha}). Since $B_i^\top\in\{B_0,\ldots,B_d\}$, there exists $k$ such that $B_k=B_i^\top$. 
Note that $A_i=\sum_{h\in\Phi_i} B_h$ for some index set $\Phi_i$ $(0\le i\le D)$. 
Thus, $A_r\circ B_k=B_k$ for some index $r$ $(0\le r\le D)$. From the definition of equitable partition and the fact the structure of intersection diagram is the same around every vertex, it follows that for any $z\in\P_i(x)$, $\partial(z,x)=r$. Now, we can conclude that \eqref{Pj} can be written as 
$$
\Delta=\{(h,r) \mid (B_i)_{xy}=1, \partial(x,y)=h,\, (B_i^\top)_{yx}=1,\,\partial(y,x)=r,\, x,y\in X,\, 0\le i\le d \}.
$$
Note that there is 1-to-1 correspondence between $\Delta$ and $\{0,1,\ldots,d\}$. For any $i$ $(0\le i\le d)$ we can produce element $(h,r)\in\Delta$, and for any element $(h,r)\in\Delta$ we can produce element $i$ $(0\le i\le d)$, i.e., it is enough to have index $i$ $(0\le i\le d)$ to derive an element $(h,r)$ of the set $\Delta$, and wise versa (with it $|\Delta|=d+1$). We also have
$$
\{R_i\}_{0\le i\le d} = \{R_{\hh}\}_{\hh\in\Delta}
$$
where
\begin{align*}
R_i&=\{(x,y)\in X\times X \mid (B_i)_{xy}=1 \}\qquad(0\le i\le d),\\
R_{\jj}&=\{(x,y)\in X\times X\mid (\partial(x,y),\partial(y,x))=\jj \}\qquad (\jj\in\Delta),
\end{align*}
i.e., for every $i$ $(0\le i\le d)$ there exists $\hh\in\Delta$ such that $R_i=R_{\hh}$, and vice versa.

Let $\oo=(0,0)$. Then, (AS1') $R_{\oo}=\{(x,x)\mid x\in X\}$; and (AS2') $\{R_{\ii}\}_{\ii\in\Delta}$ is a partition of the Cartesian product $X\times X$. Furthermore (AS3') $R^\top_{\jj}=\{(y,x)\mid (x,y)\in R_{\jj}\}$ is in $\{R_{\ii}\}_{\ii\in\Delta}$; as well as (AS4') for each triple $\ii,\jj,\hh$ $(\ii,\jj,\hh\in\Delta)$, and $(x,y)\in R_{\hh}$, the scalar
$$
|\{z\in X \mid (x,z)\in R_{\ii} \mbox{ and } (z,y)\in R_{\jj} \}|
$$
does not depend on the choice of the pair $(x,y)\in R_{\hh}$. Namely,
$$
(B_iB_j)_{xy}=|\{z\in X \mid (x,z)\in R_{\ii} \mbox{ and } (z,y)\in R_{\jj} \}|.
$$
Since $A$ generates $\M$, $(X,\{R_{\ii}\}_{\ii\in\Delta})$ is a commutative. The result follows.  


\bigskip
(ii)$\Ra$(i). Assume that $\XXi'=(X,\{R_{\ii}\}_{\ii\in\Delta})$ is a commutative $d$-class association scheme, and let $\M'$ denote the corresponding Bose--Mesner algebra. Define the set of adjacency matrices of $\XXi$ in the following way
$$
(B_{\ii})_{xy}=\left\{\begin{array}{ll}
1 & \hbox{if } \; (x,y)\in R_{\ii}, \\
0 & \hbox{if } \; (x,y)\notin R_{\ii}
\end{array}\right.\qquad(\ii\in\Delta).
$$
Note that $\{B_{\ii}\}_{\ii\in\Delta}$ is a basis of the Bose--Mesner algebra $\M'$ of $\XXi'$. For the moment let $\Phi$ denote subset of $\Delta$ with first coordinate equal to $1$, i.e., let
$$
\Phi_1=\{(1,\partial(y,x))\mid \partial(x,y)=1,\, x,y\in X \}.
$$
Note that $A\in\M'$, since $A=\sum_{\hh\in\Phi_1} B_{\hh}$. Similarly, if $\Phi_i=\{(i,\partial(y,x))\mid \partial(x,y)=i,\, x,y\in X \}$, it is not hard to see that every distance-$i$ matrix $A_i$ $(0\le i\le D)$ belong to $\M$, i.e., $A_i=\sum_{\hh\in\Phi_i} B_{\hh}$.

Since $\XXi'$ is a commutative $d$-class association scheme, $|\Delta|=d+1$, and we can enumerate the elements of $\Delta$ as $\Delta=\{\boldsymbol{0},\boldsymbol{1},\ldots,\boldsymbol{d}\}$ (where $\boldsymbol{0}=(0,0)$). The fact that $A\in\M$ yields that there exist complex scalars $w_{ij}$ $(0\le i,j\le d)$ such that
\begin{eqnarray}
 I & = & w_{00} B_{\oo} + w_{01} B_{\boldsymbol{1}} + \cdots + w_{0d} B_{\boldsymbol{d}},\nonumber\\
 A & = & w_{10} B_{\oo} + w_{11} B_{\boldsymbol{1}} + \cdots + w_{1d} B_{\boldsymbol{d}},\nonumber\\
 A^2 & = & w_{20} B_{\oo} + w_{21} B_{\boldsymbol{1}} + \cdots + w_{2d} B_{\boldsymbol{d}},\label{pm}\\
\,& \, & \vdots\nonumber\\
A^d & = & w_{d0} B_{\oo} + w_{d1} B_{\boldsymbol{1}} + \cdots + w_{dd} B_{\boldsymbol{d}}\nonumber,
\end{eqnarray}
i.e.
\begin{equation}
\label{pt}
\left [
\begin{matrix} 
I\\ A\\ A^2\\ \vdots\\ A^d
\end{matrix}
\right]=
{\underbrace{
\left[
\begin{matrix} 
w_{00} & w_{01} & ... & w_{0d}\\
w_{10} & w_{11} & ... & w_{1d}\\
w_{20} & w_{21} & ... & w_{2d}\\
\vdots & \vdots & \, & \vdots \\
w_{d0} & w_{d1} & ... & w_{dd}\\
\end{matrix}
\right]}_{=B}}
\left[
\begin{matrix} 
B_{\boldsymbol{0}}\\ B_{\boldsymbol{1}}\\ B_{\boldsymbol{2}}\\ \vdots\\ B_{\boldsymbol{d}}\end{matrix}
\right].
\end{equation}
By \eqref{pt}, $\A\subseteq\M'$. Since $B$ from \eqref{pt} is invertible (it is a change of basis matrix), we also have $\M'\subseteq\A$. The result follows.

\bigskip
(iii)$\Ra$(i). Assume that $A$ is a normal matrix, $|\Delta|=d+1$ and the number of walks from $x$ to $y$ of every given length $\ell\ge 0$ only depends on the distances $\partial(x,y)$ and $\partial(y,x)$ (and do not depend on choice of the pair $(x,y)$). For any $y,z\in X$, define a column vector $\ww(y,z)\in\CC^{d+1}$ in the following way
$$
\ww(y,z):=\Big( (A^0)_{yz}, (A^1)_{yz},\ldots,(A^d)_{yz} \Big)^{\top}.
$$
By our assumption, for any $\hh\in\Delta$ and $x,y,u,v\in X$ such that $(\partial(x,y),\partial(y,x))=(\partial(u,v),\partial(v,u))=\hh$, we have
$$
\ww(u,v)=\ww(x,y).
$$
Define the matrices $B_{\ii}$ $(i\in\Delta)$ in the following way
$$
(B_{\ii})_{xy}=\left\{\begin{array}{ll}
1 & \hbox{if } \; (\partial(x,y),\partial(y,x))=\ii, \\
0 & \hbox{otherwise }
\end{array}\right.\qquad(\ii\in\Delta).
$$
Thus, if $(B_{\ii})_{xy}=(B_{\ii})_{uv}=1$ then $\ww(x,y)=\ww(u,v)$. If $(B_{\ii})_{uv}=1$, we can write $\ww(u,v)=(w_{0i},w_{1i},\ldots,w_{di})^\top$. Let $\boldsymbol{0}=(0,0)$. By the bove comments we have that system of linear equations \eqref{pm} holds, i.e., we have
\begin{equation}
\label{pU}
\left[
\begin{matrix} 
I\\ A\\ A^2\\ \vdots\\ A^d
\end{matrix}
\right]=
\underbrace{\left[
\begin{matrix} 
w_{00} & w_{01} & ... & w_{0d}\\
w_{10} & w_{11} & ... & w_{1d}\\
w_{20} & w_{21} & ... & w_{2d}\\
\vdots & \vdots & \, & \vdots \\
w_{d0} & w_{d1} & ... & w_{dd}\\
\end{matrix}
\right]}_{=W}
\left[
\begin{matrix} 
B_{\boldsymbol{0}}\\ B_{\boldsymbol{1}}\\ B_{\boldsymbol{2}}\\ \vdots\\ B_{\boldsymbol{d}}
\end{matrix}
\right]
\end{equation}
for some complex scalars $w_{ij}$ $(0\le i,j\le d)$. Since $A$ is a normal matrix with $d+1$ distinct eigenvalues, $\{A^0,A^1,\ldots,A^d\}$ is a linearly independent set. On the other hand, $\{B_{\boldsymbol{0}},B_{\boldsymbol{1}},\ldots,B_{\boldsymbol{d}}\}$ is also a linearly independent by definition. 
For the moment let $\B=\Span\{B_{\boldsymbol{0}},B_{\boldsymbol{1}},\ldots,B_{\boldsymbol{d}}\}$ denote an algebra with respect to the elementwise--Hadamard $\circ$-product. Note that \eqref{pU} yields that $\A\subseteq\B$. On the other hand, since the matrix $W$ from \eqref{pU} is invertible (it is a change of basis matrix), we also have $\B\subseteq\A$. This yields that $\B=\A$. Now we have (AS1) $B_0=I$, the identity matrix; as well as (AS2) $\sum_{i=0}^d B_i=J$. Since $\A=\B$, every $B_i$ can be written as a polynomial in $A$, i.e., there exists some polynomial $p_i(t)\in\CC[t]$ of degree less or equal $d$ sucha that $B_i=p_i(A)$. This yields that (AS5) $B_iB_j=B_jB_i$ $(0\le i,j\le d)$. The assumption that $A$ is normal matrix, yields $A^\top\in\A$ (see, for example, \cite[Theorem~1.1]{CFGM}), so we have (AS3) $B_i^\top\in\{B_0,\ldots,B_d\}$ (recall, every $B_i$ can be written as a polynomial in $A$); and since $\A=\B$, (AS4) $B_iB_j$ is a linear combination of $B_0,B_1,\ldots,B_d$ for any $i,j$ $(0\le i,j\le d)$. The result follows.

\bigskip
(ii)$\Ra$(iii). It follows from the part (ii)$\Ra$(i) of the proof, in particular from \eqref{pm}.

\section{Further directions}
\label{ta}

The following two questions naturally arises from Section~\ref{ra}.

\begin{researchProblem}
\label{Tz}
Let $\XXi$ denote a $2$-class association scheme with vertex set $X$ and relations $\{R_0,R_1,R_2\}$. Under which algebraic-combinatorial restrictions on $\XXi$, can we split the relation $R_1$ (or $R_2$) into two (nonsymmetric) relations $R'$, $R''$ such that $(X,\{R_0,R',R'',R_2\})$ is a $3$-class association scheme.
\end{researchProblem}

Furthermore, in the language of graph theory, we can ask for something more from Research problem \ref{Tz}:

\begin{researchProblem}
Under which algebraic-combinatorial restrictions, can we `split' a strongly-regular graph $\G=\G(A)$ into two (directed) graphs $\G'=\G(A_1)$ and $\G''=\G(A_2)$, so that $A=A_1+A_2$ and $A_1$ (or $A_2$) generates a commutative $3$-class association scheme.
\end{researchProblem}

Recall that a graph which has the same intersection diagram of an equitable partition around every vertex is {\em walk-regular} (the number of closed walks of length $\ell$ rooted at vertex $x$, that is $(A^\ell)_{xx}$, only depends on $\ell$, for each $\ell\ge 0$). If we have an equitable partiton of an undirected walk-regular graph, we can make an intersection diagram, and using this intersection diagram we can compute quotient matrix $B$. Then, from \cite[Proposition~4.1]{DF} it follows that every eigenvalue of $\G$ is also an eigenvalue of $B$. The problem that we are struggling with (it can be a possible further direction in the topic of this paper) is proving or disproving the claim given in Research problem~\ref{pe}.

\begin{researchProblem}
\label{pe}
Let $\M$ denote the Bose--Mesner algebra of a commutative $d$-class association scheme $\XXi=(X,\R)$, $A\in\M$ denote a $01$-matrix and let $\G=\G(A)$ denote a (directed) graph with the adjacency matrix $A$. If the following {\rm (i)--(iii)} hold
\begin{enumerate}[label=\rm(\roman*)]
\item for every vertex $x\in X$, there exists an $x$-distance-faithful intersection diagram of an equitable partition $\Pi_x$ with $d+1$ cells;
\item the structure of the $x$-distance-faithful intersection diagram of the equitable partition $\Pi_x$ from {\rm (i)} does not depend on $x$;
\item the graph $\G$ does not have an $x$-distance-faithful intersection diagram with less than $d+1$ cells (i.e., $d+1$ is the smallest number of cells for which there exists an $x$-distance-faithful equitable partition);
\end{enumerate}
prove or disprove that then $\G$ generates $\M$.
\end{researchProblem}

Let $r$ denote some natural number and for the moment consider properties (i) and (ii) of the given research problem above. It is not hard (but also not so easy) to find a graph which has the same intersection diagram (of an equitable partition) with $r+1$ cells around every vertex, but which has less than $r+1$ distinct eigenvalues. In the spirit of this paper, this intuitively yields that such a graph will not generate a commutative association scheme. One such example is so-called `chordal ring' $(12, 4)$ of prism shown in \cite[Figure~3]{FQpG}. The spectrum of this graph is $\{[3]^1,[2]^2,[1]^1,[0]^4,[-1]^1,[-2]^2,[-3]^1\}$, and the spectrum of its $x$-distance-faithful intersection diagram is $\{[3]^1,[2]^1,[1]^1,[0]^2,[-1]^1,[-2]^1,[-3]^1\}$ (has $8$ cells and $7$ distinct eigenvalues).

A solution to the above research problem will give a combinatorial property of a (directed) graph that generates a commutative association scheme; in this way we might classify commutative association schemes with respect to the above property.

\section*{Acknowledgments}

This work is supported in part by the Slovenian Research Agency (research program P1-0285).

\section*{Declaration of competing interest}

The authors declare that they have no known competing financial interests or personal relationships that could have appeared to influence the work reported in this paper.


{\small
\bibliographystyle{references}
\bibliography{associationSch_v2}
}

\end{document}